\documentclass[a4paper,12pt]{amsart}
\usepackage{microtype,mathtools, a4wide, color, hyperref}
\usepackage[utf8]{inputenc}
\usepackage[alphabetic,backrefs]{amsrefs}

\renewcommand{\Re}{\operatorname{Re}}

\newcommand{\C}{\mathbb{C}}
\newcommand{\R}{\mathbb{R}}
\DeclareMathOperator{\diff}{d\!}
\DeclareMathOperator{\supp}{supp}
\newcommand{\Tr}{\operatorname{Tr}}

\newtheorem{lemma}{Lemma}[section]
\newtheorem{thm}{Theorem}[section]
\newtheorem{prop}{Proposition}[section]
\newtheorem{cor}{Corollary}[section]

\theoremstyle{remark}
\newtheorem*{rmk}{Remark}

\begin{document}
\author{María Ángeles García-Ferrero}
\address{Instituto de Ciencias Matem\'aticas CSIC-UAM-UC3M-UCM, c/ Nicol\'as Cabrera 13--15,
28049 Madrid, Spain}
\email{\href{mailto:garciaferrero@icmat.es}{\texttt{garciaferrero@icmat.es}}}
\author{Joaquim Ortega-Cerdà}
\address{Dept.\ Matem\`atica i Inform\`atica,
 Universitat  de Barcelona,
Gran Via 585, 08007 Bar\-ce\-lo\-na, Spain and
CRM, Centre de Recerca Matemàtica, Campus de Bellaterra Edifici C, 08193 
Bellaterra, Barcelona, Spain}
\email{\href{mailto:jortega@ub.edu}{\texttt{jortega@ub.edu}}}

\title{Stability of the concentration inequality on polynomials}

\thanks{MAGF has been partially supported by the grants CEX2023-001347-S, PID2021-125021NAI00, PID2021-124195NB-C32, PID2021-122154NB-I00 and PID2021-122156NB-I00, funded by MCIN/AEI/10.13039/501100011033.
JOC has been partially supported by grants  
PID2021-123405NB-I00, MDM-2014-0445 by the Spanish Research Agency (AEI) 
and by the Departament de Recerca i Universitats, grant 2021 SGR 00087.}

\begin{abstract}
 In this paper, we study the stability of the concentration inequality for one-dimensional complex polynomials.  We provide the stability of the local concentration inequality and a global version using a Wehrl-type entropy. 
\end{abstract}

\maketitle

\section{Introduction}
\label{sec:intro}
The Paley--Wiener space consists of square integrable functions $f\in L^2(\R)$ that are 
band-limited, i.e. $\supp \hat f \subset [-\pi/2,\pi/2]$.
The well-known Donoho--Stark conjecture \cite{DonSta89} states that, among all 
functions $f$ in the Paley--Wiener space and all measurable subsets 
$\Omega$ of the real line with fixed Lebesgue measure $|\Omega|= \ell$, the concentration 
operator 
\[C_\Omega(f) := \frac{\int_{\Omega} |f(x)|^2\, dx}{\int_{\R} |f(x)|^2\, dx  }\]
achieves a maximum when $\Omega$ is an interval, i.e. $\Omega=[a-\ell/2,a+\ell/2]$.
The conjecture has been proved in \cite{DonSta93}  provided $\ell<\frac{0.8}{\pi}$, but the general case remains an open conjecture. 

A natural finite-dimensional analogous 
problem is to replace the band-limited functions with polynomials of  bounded 
degree endowed with a suitable $L^2$ norm.
Let us consider the space $\mathcal P_N$  of polynomials 
of degree less than or equal to $N$. 
If $z= x+i y\in \C$,  $dm(z) = \frac{dx\wedge dy}{\pi (1+|z|^2)^2}$ 
defines a probability measure on $\C$,  which is the push-forward of the normalized Lebesgue measure in the sphere of radius $\frac{1}{2\sqrt{\pi}}$ in $\R^3$  by the stereographic projection.
We endow $\mathcal P_N$ with the Hermitian product given by
\[
 \langle P, Q \rangle_{N}  \coloneqq (N+1)\int_{\C} 
\frac{P(z)\overline{Q(z)}}{(1+|z|^2)^{N}} 
dm(z), \quad P, Q\in\mathcal P_N.
\]
The corresponding norm for $P\in \mathcal P_N$ is  defined as
\[
 \|P\|_{N}^2\coloneqq (N+1)\int_{\C} \frac{|P(z)|^2}{(1+|z|^2)^{N}} 
dm(z).
\]
The normalizing factor $(N+1)$ is chosen so that $\|1\|_N = 1$.

We define the concentration operator in $\mathcal P_N$ for any measurable set $\Omega\subset \C$ and any $P\in \mathcal P_N$  as
\[
C_{N,\Omega}(P) := \frac{(N+1)\int_{\Omega} \frac{|P(z)|^2}{(1+|z|^2)^{N}} 
dm(z)}{\|P\|^2_{N}}.
\]
The problem analogous to the one in the Donoho--Stark conjecture is to find the maximum of 
$\sup_{P\in \mathcal P_N}C_{N,\Omega}(P)$ among all measurable sets $\Omega\subset \C$ such 
that $m(\Omega) = \ell$. This was accomplished in \cite{KNOCT} and \cite{Frank23}. The 
maximum is achieved when $\Omega$ is the disc centered at the origin of measure $\ell$ 
and $P$ is a constant, i.e.
\begin{align}\label{eq:qualitative}
    C_{N,\Omega}(P)\leq  C_{N,\Omega^*}(1),
\end{align}
where $\Omega^*$ is the disc centered at the origin with $m(\Omega^*)=m(\Omega)=\ell$. 
All the other extremal sets $\Omega$ are discs (with respect to the chordal distance,  which are also Euclidean discs on the plane) of measure $\ell$, for which $C_{N,\Omega}(P)$ achieves the maximum when $P$ is a multiple of the reproducing kernel, which we describe below, at the center 
of the disc (the chordal disc center, not the Euclidean center).

Let $e_n(z)\coloneqq\sqrt{\binom{N}{n}}z^n$ for $n=0,1,\dots, N$. With the normalization of the Hermitian product that we have picked, $ 
\{e_n(z)\}_{n=0}^N$ is an orthonormal 
basis of $\mathcal P_N$.
Therefore, $\mathcal P_N$ endowed with
the Hermitian product is a reproducing kernel Hilbert space with kernel 
$$k_N(z,\zeta) = \sum_{n= 0}^N e_n(z)\overline{e_n(\zeta)} = (1+z\bar \zeta)^N.$$  
Thus, for all $P\in \mathcal P_N$ and all
$z\in \C$
$$P(z) = \langle P, k_N(\cdot,z)\rangle_{N}=(N+1)\int_{\mathbb C} \frac{(1+z\bar\zeta)^N P(\zeta)}{(1+|\zeta|^2)^N}\, 
dm(\zeta).$$
By the extremal property of the reproducing kernel we have that
$$\sup_{z\in \mathbb C} \frac{|P(z)|^2}{(1+|z|^2)^N} \le \|P\|_{N}^2.$$
We finally denote by $\kappa_{N,\zeta}$ the normalized reproducing kernels, given by 
\begin{align*}
    \kappa_{N,\zeta}(z) =\frac{k_N(z,\zeta)}{\|k_N(\cdot,\zeta)\|_{N}}= \frac{(1+ z\bar \zeta)^N}{(1+|\zeta|^2)^{N/2}}.
\end{align*}

Our aim is to study the stability of the concentration inequality \eqref{eq:qualitative}. That is, we 
want to prove that whenever we have a measurable set 
$\Omega\subset \mathbb C$ and a polynomial $P\in \mathcal P_N$ with norm one, if the concentration $C_{N,\Omega}(P)$  is close to the 
maximal among all sets of  measure $m(\Omega)$, then $\Omega$ and $P$ must be close to a 
disc and to  a normalized reproducing kernel, respectively. 

As we can see from the form of the reproducing kernel $k_N(z,\zeta)$, the space of rescaled polynomials 
$\mathcal P_N$ resembles as $N\to\infty$ the Fock space $\mathcal F^2$ of 
entire functions such that $\int_{\C} |f(z)|^2 e^{-\pi |z|^2}dz <+\infty$. The 
reproducing kernel for such space is $k(z,\zeta) = e^{z\bar\zeta}$. 

The analog to \eqref{eq:qualitative} in the Fock space is proved in \cite{NT22}, while its stability is well studied  in \cite{GGRT}. These results may  be seen in terms of the energy concentration for the short-time Fourier transform (STFT) with the Gaussian window. They can also be interpreted as a quantitative Faber-Krahn inequality for the localization operator defined in terms of the STFT. 

Our stability estimates are modeled after the ones in \cite{GGRT} for the Fock space, but there are some points where 
necessarily our results are technically more delicate. 
Formally, the results in 
the Fock space can be obtained from the results in the space of polynomials. 
This is carried out in detail in Section~\ref{sec:Fock}.

In order to state our results, we define the distance of  any $P\in\mathcal P_N$ with 
$\|P\|_{N}=1$ to the normalized reproducing 
kernels in $\mathcal P_N$ as 
\begin{align}\label{eq:defDN}
    D_N(P)=\min\left\{\left\| P-e^{i\theta}\kappa_{N,a}\right\|_{N}: a\in\C, \theta\in [0,2\pi] \right\}.
\end{align}
Notice that $D_1(P)=0$ for all $P\in\mathcal P_1$.  The following statements also hold for $N=1$ but the proofs are immediate, so from now on we will assume $N\geq 2$.

Our first result is the stability of \eqref{eq:qualitative}, which can be also read as how close a polynomial of unit norm is to the normalized reproducing kernels if its concentration is close to the maximal one:
\begin{thm}\label{thm:concentration}
    There exists a constant $C>0$ (independent of $N$) such that for any measurable set $\Omega\subset\C$ with positive measure and any $P\in\mathcal{P}_N$ with $\|P\|_{N}=1$, there holds
    \begin{align*}
        C_{N,\Omega}(P)\leq \left(1-C\left(1-m(\Omega)\right)^{N+1}D_N(P)^2\right) C_{N,\Omega^*}(1),
    \end{align*}
    where $\Omega^*$ is the disc centered at $z=0$ with $m(\Omega^*)=m(\Omega)$.
    Equivalently, 
    \begin{align*}
        D_N(P)
        \leq \Big(C^{-1} \left(1-m(\Omega)\right)^{-(N+1)}\delta_N(P, \Omega)\Big)^{1/2},
    \end{align*}
    where
    \begin{align}\label{eq:defdeltaN}
    \begin{split}
        \delta_N(P,\Omega)
        &=
        1-\frac{C_{N,\Omega}(P)}{C_{N,\Omega^*}(1)}=1-\frac{N+1}{1-\left(1-m(\Omega)\right)^{N+1}}\int_\Omega \frac{| P(z)|^2}{(1+|z|^2)^N}\, dm(z).
        \end{split}
    \end{align}
\end{thm}

Observe that $\delta_N(P,\Omega)$ is the \emph{combined deficit} in the parlance of \cite{GGRT}, which measures how close $P$ is to be an optimal polynomial for the concentration in the domain $\Omega$. 

We now turn our attention to quantify the closeness of $\Omega$ to the extremal sets of the concentration.
In order to do that, we introduce a measure of such closeness:
Let $\Omega_1, \Omega_2\subset \C$ be two measurable sets such that $m(\Omega_1)=m(\Omega_2)$. We  define the following distance between them:
\begin{align*}
    \mathcal A_m(\Omega_1,\Omega_2)  =\frac{m(\Omega_1\backslash\Omega_2)+m(\Omega_2\backslash\Omega_1)}{m(\Omega_1)}.
\end{align*}
The Fraenkel asymmetry of a set $\Omega\subset\C$ measures its $\mathcal A_m$-distance to the closest disc of the same measure, i.e., 
\begin{align*}
    \mathcal A_m(\Omega)=\inf\big\{\mathcal A_m\big(\Omega,\mathcal D_r(z)\big): \ m\big(\mathcal D_r(z)\big)=m(\Omega), z\in\C\big\},
\end{align*}
where 
\begin{align*}
    \mathcal D_r(z)=\left\{w\in\C: d(z,w)=\frac{|z-w|}{\sqrt \pi\sqrt{(1+|z|^2)(1+|w|^2)}}\leq r\right\}.
\end{align*}
Here $d(\cdot,\cdot)$ is the chordal distance. Notice that $\mathcal D_r(z)=\mathbb D_ {\rho}(\zeta)$ for some $\rho$ and $\zeta$, where   $\mathbb D_\rho(\zeta)$ denotes the usual disc in the Euclidean metric centered at $\zeta$. If $z=0$, then $\zeta=0$ and $\rho=\frac{\sqrt{\pi}r}{\pi\sqrt{1-\pi r^2}}$. 

\begin{prop}\label{prop:domains}
Under the same assumptions of Theorem~\ref{thm:concentration}, for all $N\in\mathbb N$ we have
    \begin{align*}
        \mathcal A_m(\Omega)\leq  C\frac{\big(1-m(\Omega)\big)^{-3 (N+1)/2}}{m(\Omega)}\delta_N(P,\Omega)^{1/2}.
    \end{align*}
\end{prop}

We finally study similar stability results for the measure of the concentration of the reproducing kernels in terms of the  Wehrl entropy.  
Namely, for any $P\in \mathcal P_N$, its  Wehrl entropy is defined as
\begin{align*}
    S_N(P)=-(N+1)\int_{\C} \frac{|\hat P(z)|^2}{(1+|z|^2)^N}\log\left(\frac{|\hat P(z)|^2}{(1+|z|^2)^N}\right)dm(z),
\end{align*}
where $\hat P(z) = P(z)/\|P\|_{N}$.
In addition, given  a convex, non-linear, continuous function $\Phi:[0,1]\to\mathbb{R}$ with $\Phi(0)=0$,  we can define a generalized Wehrl entropy as follows: 
\begin{align*}
    S_{N,\Phi}(P)=-(N+1)\int_{\C} \Phi\left(\frac{|\hat P(z)|^2}{(1+|z|^2)^N}\right)dm(z).
\end{align*}

In \cite{Lieb78} it was conjectured that $S_N(P)$  is minimized when $P$ is a reproducing kernel. 
This was proved in \cite{LiebSo14} (see also \cite{LiebSo16}), after some partial results in \cites{Sch99, Bod05}. Furthermore,  the reproducing kernels are the unique minimizers. that was  shown independently in \cite{KNOCT} and \cite{Frank23}.

The next result quantifies the distance of $P$ to the reproducing kernels in terms of the difference of its generalized Wehrl entropy to its minimum value.

\begin{thm}\label{thm:Wehrl}
    Let $\Phi:[0,1]\to\mathbb{R}$ be a non-linear, convex, continuous function with $\Phi(0)=0$. Then there exists a constant $C>0$ (depending only on $\Phi$ and not on $N$) such that for any $P\in\mathcal P_N$ with $\|P\|_{N}=1$, it holds
    \begin{align*}
        D_N(P)^2 \le C\left(S_{N,\Phi}(P)-S_{N,\Phi}(1)\right).
    \end{align*}
\end{thm}

The analog result in the Fock space is studied in \cite{FNT23}, which quantifies the inequality in \cites{Lieb78, Car91, LiebSo14} (see also \cite{Luo00}, and \cites{KNOCT, Frank23} for the uniqueness of the minimimizers), which in turn answered positively to the Wehrl's conjecture in \cite{Wehrl}. 

As in the case of the previous estimates,
these results can be formally inferred from ours in the space of polynomials (see Section~\ref{sec:Fock}).
Theorem~\ref{thm:Wehrl} follows the scheme in \cite{FNT23}, but again, our proof faces  some additional difficulties. 

We notice that all the previous results are  optimal regarding the power of the distance to the reproducing kernels $D_N(P)$  or of the deficit $\delta(P,\Omega)$. 
See Section~\ref{sec:optimality} for more details. 

We finally remark that the preceding statements are also valid for operators $\mathcal P_N\to\mathcal P_N$ that are positive-semidefinite and have unit trace. In the same spirit the polynomials  define  pure states (and in particular, coherent states in the case of normalizing reproducing kernels), these operators turn out to express mixed states. For more precise details, see Section~\ref{sec:generalop}.

The rest of the article is organized as follows:
In Section~\ref{sec:superlevelsets} we state and prove some technical lemmas in which we estimate the measure of the superlevel sets of the weighted polynomials.
Section~\ref{sec:proofs} is devoted to the proofs of the main results. 
In Section~\ref{sec:Fock} we recover the results for the Fock space by taking the limit as $N\to \infty$ of our main results. 
Section~\ref{sec:optimality} focuses on the sharpness of the stability estimates, and in Section~\ref{sec:generalop} we collect the results and the proofs for general operators. 
In Section~\ref{sec:Schatten} we conclude with some remarks on the concentration operator in Schatten $p$-spaces.

\section{On the measure of super-level sets}
\label{sec:superlevelsets}

Given $P\in \mathcal{P}_N$, we introduce the following notation:
\begin{align*}
    u(z)& = \frac{|P(z)|^2}{(1+|z|^2)^N},
    \\
    T&=\sup_{z\in\mathbb C} u(z),\\
    \mu(t)& = m(\{u(z) > t\})
\end{align*}
We notice that if $P=1$, then $T=1$ and
\begin{align}\label{eq:mu0}
    \mu(t)=(1-t^{1/N})=:\mu_0(t).
\end{align}
Moreover 
\[|P(z)|^2 = |\langle P, k_N(\cdot, z)\rangle_{N}|^2\le \|P\|_N^2 (1+|z|^2)^N.\]
Therefore, if $\|P\|_N = 1$, then $T \le 1$ and the equality is attained only when $P$ is a unimodular constant times a normalized reproducing kernel.

This section is devoted to the study of the functions $\mu(t)$ and its relationship with $\mu_0(T)$. The results and their proofs mimic those in Section~2 of \cite{GGRT} in the Fock space, which can be recovered from ours if $N\to\infty$.
Nevertheless, our proofs are more intricate, since we do not only need to carry the dependence on $N$ of all the estimates, but to work with the  measure $dm(z)$.
Moreover, while in \cite{GGRT} they can inherit some estimates on $\mu(t)$ from \cite{NicRic}, we need to introduce here Lemma~\ref{lem:on_mu}, where related estimates are provided.

\begin{lemma}\label{lemma:super-level} 
For every $t_0\in (0,1)$, there exists a threshold $T_0\in (t_0,1)$ and a 
constant
$C_0=C_0(t_0)>0$ with the following property:
If $P\in \mathcal{P}_N$ is such that $\|P\|_{N}=1$ with
$T\geq T_0$,
then
\begin{equation}
    \label{newestmu}
     \mu(t)\leq  \left( 1+{C_0
(1-T)}\right)\left (1-\left(\frac{t}T\right)^{1/N}\right)\quad\forall t\in
[t_0,T].
\end{equation}
\end{lemma}

\begin{proof}
The proof is split into five steps.

\noindent\textit{Step 1: Decomposition of $P$.} 
We can assume without loss of generality  that $u(z)$ attains its supremum at $z=0$, and that, in particular,  $P(0)=\sqrt{T}$. This in addition implies that $P'(0)=0$, i.e. $\langle P, e_1\rangle_N=0$.
We then write 
\begin{align*}
    P(z)=\sqrt{T}+\varepsilon Q(z), 
\end{align*}
where   
\begin{align}\label{eq:defQ}
    Q(z)=\sum_{n=2}^N q_n e_n(z), \quad \|Q\|_{N}^2=\sum_{n=2}^N|q_n|^2=1.
\end{align}
The assumption on the norm of $Q$ implies
\begin{align}\label{eq:epsilonT}
    \varepsilon^2
    &=\|P-\sqrt{T}\|_{N}
    =1+T-2\sqrt{T}\Re \langle P,1\rangle_{N}=1-T.
\end{align}

With the previous decomposition, we obtain
\begin{equation}
  \label{eq:estimateP^2}
  |P(z)|^2\leq T+ \varepsilon^2 |Q(z)|^2+ 2\sqrt{T}\varepsilon \Re Q(z)
\end{equation}
By \eqref{eq:defQ} and the Cauchy-Schwarz inequality, we estimate $Q(z)$ as  follows:
\begin{equation}
\label{eq:estimateQ^2}
|Q(z)|^2\leq \left(\sum_{n=2}^N |q_n|^2\right)
\left(\sum_{n=2}^N |e_n(z)|^2\right)
=\sum_{n=2}^N\binom{N}{n}|z|^{2n}
= (1+|z|^2)^N-1-N |z|^2.
\end{equation}
In particular, 
\begin{equation}
\label{eq:estimateQ^2_s}
|Q(z)|^2\leq (1+|z|^2)^N-1.
\end{equation}

\noindent\textit{Step 2:  Estimates for $\Re Q$.}
Throughout this step, we are going to use  the inequalities
\begin{align}\label{eq:ineqbinomial}
  \binom{N}n  \leq \frac{N^2}{n(n-1)}\binom{N}{n-2},  \quad {\binom{N}n  \leq \frac{N^2}{{n(n-1)}}\binom{N-2}{n-2},}\quad n\geq 2  
\end{align}
First of all, arguing as in \eqref{eq:estimateQ^2} and using the first inequality in \eqref{eq:ineqbinomial}
\begin{align*}
    |Q(z)|^2\leq 
    \sum_{n=2}^N\binom{N}{n}|z|^{2n}
    \leq \frac{N^2}{2} \sum_{n=2}^{N}\binom{N}{n-2}|z|^{2n}\leq \frac{N^2}{2} |z|^4(1+|z|^2)^N.
\end{align*}
Differentiating $Q$ and using similar arguments, we infer
\begin{align*}
    |Q'(z)|^2
    &\leq 
    \sum_{n=2}^N n^2 \binom{N}{n}|z|^{2(n-1)}
    \leq N^2 \sum_{n=2}^{N}\binom{N}{n-2}|z|^{2(n-1)}\leq N^2 |z|^2(1+|z|^2)^N.
\end{align*} 
{
We differentiate $Q$ again and we make use of the second inequality in \eqref{eq:ineqbinomial}  to obtain
\begin{align*}
    |Q''(z)|^2
    &\leq 
    \sum_{n=2}^N n^2(n-1)^2 \binom{N}{n}|z|^{2(n-2)}
    \leq N^2 \sum_{n=2}^N {n^2} \binom{N-2}{n-2}|z|^{2(n-2)}.
\end{align*} 

If $N\geq 4$ we can use again the first inequality  in \eqref{eq:ineqbinomial}  to bound the last term as follows:
\begin{align*}
    \sum_{n=2}^N n^2\binom{N-2}{n-2}|z|^{2(n-2)}
    &\leq  C (1+N|z|^2) + N^2 \sum_{n=4}^N \binom{N-2}{n-4}|z|^{2(n-2)}
    \\&\leq C (1+N^2 |z|^4) + N^2|z|^4(1+|z|^2)^{N-2}.
\end{align*}
Notice that if $N\leq 4$, the same estimate holds. 
Hence, 
\begin{align*}
    |Q''(z)|^2
    &\leq C N^2(1+N|z|^2)^2(1+|z|^2)^{N-2}.
\end{align*}
}

Let $h(z):=\Re Q(z)$. Since $|h(z)|\leq |Q(z)|$, we have
\begin{align}\label{eq:estimateh}
    |h(z)|\leq \frac{N}{\sqrt{2}} |z|^2(1+|z|^2)^{N/2}.
\end{align}
Furthermore, the Cauchy-Riemann equations imply that  $|\nabla h(z)|= |Q'(z)|$
and  
\[|D^2 h(z)|= \sqrt 2 |Q''(z)|.\]
Hence, one gets the following estimates  for
the first and second
radial derivatives of $h(r e^{i\theta})$, which are independent of the angular variable: 
\begin{align}\label{eq:estimatehr}
\left\vert  {\partial_r h(r e^{i\theta})} \right\vert\leq
 Nr(1+r^2)^{N/2},
\\
\label{eq:estimatehrr}
\left\vert  {\partial_{rr} h(r e^{i\theta})}\right\vert
\leq 
{C N (1+Nr^2) (1+r^2)^{N/2-1}}.
\end{align}

\noindent\textit{Step 3: Star-shaped domains in $\{u(z)>t\}$.} 
From \eqref{eq:estimateP^2} and \eqref{eq:estimateQ^2_s}, we have 
\begin{align*}
    u(re^{i\theta})=\frac{|P(re^{i\theta})|^2}{(1+|z|^2)^N}\leq \frac{T -\varepsilon^2+2\sqrt{T}\varepsilon h(re^{i\theta})}{(1+r^2)^N}+\varepsilon^2.
\end{align*}
Then
\begin{align*}
    \mu(t)=m\big(\{u(re^{i\theta})>t\}\big)\leq m\big(\{(t-\varepsilon^2)(1+r^2)^N-2\sqrt{T}\varepsilon h(re^{i\theta})< T-\varepsilon^2\}\big).
\end{align*}

Let $s\in[0,1]$ be an extra variable and define for any $\theta\in[0, 2\pi)$ the function
\begin{align*}
    g_\theta(r,s)=\frac{t-\varepsilon^2}{T-\varepsilon^2}(1+r^2)^N-\frac{2\sqrt{T}\varepsilon}{T-\varepsilon^2}sh(re^{i\theta}).
\end{align*}
We also introduce the sets
\begin{align*}
    E_s=\{re^{i\theta}\in\C: g_\theta(r,s)<1\}.
\end{align*}
Therefore, as far as $T-\varepsilon^2=2T-1>0$, i.e. $T>\frac 12$, we have
\begin{align}
    \mu(t)\leq m(E_1).
\end{align}
In order to estimate $m(E_1)$ 
we will see that the sets $E_s$ are star-shaped domains with  $E_s\subset E_{s'}$ for $s<s'$.
Finally, using that $h$ is a real-valued harmonic polynomial and the results in Step 2, we will be able to estimate $m(E_s)$ in terms of $m(E_0)$.

The fact that $E_s$ is star-shaped with respect to the origin follows from $g_\theta(0,s)<1$ and $\partial_r g_\theta(r,s)>0$ for any $s\in(0,1)$ and $\theta\in[0, 2\pi)$. 
Indeed, $g_\theta(0,s)=\frac{t-\varepsilon^2}{T-\varepsilon^2}< 1$.
Moreover, using \eqref{eq:estimatehr} and recalling that $N\geq 2$, 
\begin{align}
    \label{eq:gr}
    \partial_r g_\theta(r,s)
    &=2N\frac{t-\varepsilon^2}{T-\varepsilon^2}r(1+r^2)^{N-1}-\frac{2\sqrt{T}\varepsilon}{T-\varepsilon^2}s\partial_r h(re^{i\theta})\\\nonumber
    &\geq 2Nr(1+r^2)^{N-1}\left(\frac{t-\varepsilon^2}{T-\varepsilon^2}-\frac{\sqrt{T}\varepsilon}{T-\varepsilon^2}s(1+r^2)^{1-N/2}\right)\\\nonumber
    &\geq 2Nr(1+r^2)^{N-1}\frac{t-\varepsilon^2-\sqrt{T}\varepsilon}{T-\varepsilon^2}.
\end{align}
Given $t_0\in(0,1)$, let $T>\max\big\{t_0, \frac12\big\}$ such that 
\begin{align}\label{eq:T0}
    \varepsilon^2+\sqrt{T}\varepsilon=1-T+\sqrt{T(1-T)}\leq\frac{t_0}{2}.
\end{align}
Then, for any $t\in(t_0, T)$,
\begin{align}\label{eq:estimategr}
    \partial_r g_\theta(r,s)
    &\geq \frac{t_0}{T-\varepsilon^2}Nr(1+r^2)^{N-1}>0.
\end{align}

\noindent\textit{Step 4: Estimates for the radial distance of $E_s$.}
Given $s\in[0,1]$ and $\theta\in[0,2\pi)$, let $r_s(\theta)>0$ be the unique solution of $g_\theta(r,s)=1$.
Clearly, $r_0(\theta)=r_0$ given by $(1+r_0^2)^N=\frac{T-\varepsilon^2}{t-\varepsilon^2}$. In addition, using  \eqref{eq:estimategr} and since $T-\varepsilon^2<1$, one obtains
\begin{align*}
    g_\theta(r,s)
    &\geq g_\theta(0,s)+\int_0^rt_0N\rho(1+\rho^2)^{N-1}d\rho
    \geq \frac{t_0}{2}(1+r^2)^N, 
\end{align*}
so $(1+r_s(\theta)^2)^N\leq \frac{2}{t_0}$.

Applying the implicit function theorem, we have
\begin{align}
\label{eq:rs}
    \partial_s r_s=-\frac{\partial_s g_\theta(r_s,s)}{\partial_r g_\theta(r_s, s)}=\frac{2\sqrt{T}\varepsilon}{T-\varepsilon^2}\frac{h(r_se^{i\theta})}{\partial_r g_\theta(r_s, s)}.
\end{align}
Here and in the remaining of the proof,  the explicit dependence on $\theta$ of $r_s$ is omitted.
Differentiating again with respect to $s$ and taking into account that $g_\theta$ depends linearly on $s$, it follows that
\begin{align}
\label{eq:rss}
\begin{split}
    \partial_{ss} r_s
    &=-\frac{\partial_{rs} g_\theta(r_s,s)\partial_s r_s}{\partial_r g_\theta(r_s, s)}-\frac{\partial_{s} r_s}{\partial_r g_\theta(r_s, s)}\left(\partial_{rr}g_\theta(r_s,s) \partial_s r_s+\partial_{rs}g_\theta(r_s,s)\right)
   \\& =-\frac{2\partial_{rs} g_\theta(r_s,s)\partial_s r_s+\partial_{rr} g_\theta(r_s,s)(\partial_s r_s)^2}{\partial_r g_\theta(r_s, s)}.
   \end{split}
\end{align}

Using \eqref{eq:estimateh} and \eqref{eq:estimategr} in \eqref{eq:rs}, we finally obtain
\begin{align}
\label{eq:estimaters}
    |\partial_s r_s|\leq \frac{\sqrt{2T}\varepsilon}{t_0}r_s(1+r_s^2)^{1-N/2}.
\end{align}
Let $\zeta(s)=\frac{r_s^2}{1+r_s^2}$ and notice that
\begin{align*}
    \frac{|\partial_s\zeta(s)|}{\zeta(s)}
    &=\frac{2|\partial_s r_s|}{r_s(1+r_s^2)}\leq 2\frac{|\partial_sr_s|}{r_s}
    \leq  2\frac{\sqrt{2T}\varepsilon}{t_0}(1+r_s^2)^{1-N/2}\leq \sqrt{2}\leq \log 4
\end{align*}
where in the last-but-one step we have used  \eqref{eq:T0}.
Hence, for any $s\in[0,1]$,
\begin{align*}
    \log\frac{\zeta(s)}{\zeta(0)}=\int_0^s \frac{\partial_s\zeta(\sigma)}{\zeta\sigma)}d\sigma\leq  \log 4
\end{align*}
and therefore,
\begin{align}\label{eq:rsvsr0}
    \frac{r_s^2}{1+r_s^2}\leq 4 \frac{r_0^2}{1+r_0^2}.
\end{align}

By \eqref{eq:estimatehr} and \eqref{eq:estimatehrr}, we observe
\begin{align*}
|\partial_{rs}g_\theta(r,s)|
    &=\frac{2\sqrt{T}\varepsilon}{T-\varepsilon^2}|\partial_r h(re^{i\theta})|
    \leq \frac{2\sqrt{T}\varepsilon}{T-\varepsilon^2} Nr(1+r^2)^{N/2},\\
    |\partial_{rr}g_\theta(r,s)|
    &\leq 
    2N\frac{t-\varepsilon^2}{T-\varepsilon^2}\Big((1+r^2)^{N-1}+2(N-1)r^2(1+r^2)^{N-2}\Big)+\frac{2\sqrt{T}\varepsilon}{T-\varepsilon^2}s\left|\partial_{rr} h(r e^{i\theta})\right|\\
    &\leq {
    \frac{C}{T-\varepsilon^2}N(1+Nr^2)(1+r^2)^{N-2}}.\end{align*}
Using   these estimates, \eqref{eq:estimategr} and \eqref{eq:estimaters} in \eqref{eq:rss}, we   get
\begin{align*}
    |\partial_{ss} r_s|&\leq C\frac{T{\varepsilon^2}}{t_0^2}\left({(1+r_s^2)^{2-N}+\frac{1}{t_0}(1+r_s^2)^{1-N}(1+Nr^2)}\right)r_s.
\end{align*}
{Taking into account that the function  $(1+x)^{1-N}(1+Nx)$ attains its maximum at $x=(N(N-2))^{-1}$ for $N\geq 2$ and therefore can be bounded by $2$, we finally obtain  }
\begin{align}\label{eq:estimaterss}
    |\partial_{ss} r_s|& \leq C\frac{T\varepsilon^2}{t_0^3}Nr_s.
\end{align}

\noindent\textit{Step 5: Estimate of $m(E_1)$.}
Since $E_s$ is a star-shaped domain,  we can write its measure in terms of $r_s(\theta)$ as follows
\begin{align*}
    M(s):=m(E_s)=\frac{1}{2\pi}\int_0^{2\pi}\frac{r_s(\theta)^2}{1+r_s(\theta)^2}d\theta.
\end{align*}
Recalling that $r_s(\theta)$ is uniformly bounded and so are  $|\partial_sr_s(\theta)|$ and $|\partial_{ss}r_s(\theta)|$ according to \eqref{eq:estimaters} and \eqref{eq:estimaterss}, respectively, we can differentiate $M(s)$ under the integral obtaining 
\begin{align*}
    M'(s)
    &=\frac {1}\pi \int_0^{2\pi} \frac{r_s\partial_s r_s}{(1+r_s^2)^2}d \theta,\\
    M''(s)
    &=\frac {1}\pi \int_0^{2\pi}\left( \frac{(1-3r_s^2)(\partial_s r_s)^2}{(1+r_s^2)^3}+\frac{r_s\partial_{ss} r_s}{(1+r_s^2)^2}\right)d \theta.
\end{align*}

On the one hand, $M'(0)=0$. Indeed, by \eqref{eq:rs}, \eqref{eq:gr} and recalling that $r_0(\theta)=r_0$, we have
\begin{align*}
    M'(0)=\frac{C_{T, N}}{(1+r_0^2)^{N+2}} \int_0^{2\pi} {h(r_0e^{i\theta})}d \theta,
\end{align*}
where $C_{T,N}$ is a constant depending on  $T$ and $N$.
By the mean value theorem for the (harmonic) function $h$, we know $\int_0^{2\pi} {h(r_0e^{i\theta})}d \theta=2\pi r_0 h(0)=0$, which concludes the proof of the claim.

On the other hand, using \eqref{eq:estimaters} and \eqref{eq:estimaterss} together with \eqref{eq:rsvsr0}, we can bound $M''(s)$ as follows
\begin{align*}
    |M''(s)|
    \leq C\int_0^{2\pi} \frac{|\partial_sr_s|^2+r_s|\partial_{ss}r_s|}{1+r_s^2}d\theta
    \leq  C\frac{T\varepsilon^2}{t_0^3} \int_0^{2\pi} \frac{r_s^2}{1+r_s^2} \diff \theta
    \leq C\frac{T\varepsilon^2}{t_0^3} M(0).
\end{align*}

Combining the previous observations with the Taylor's formula for $M(s)$,
we conclude that 
\begin{align}\label{eq:estimateMs}
    M(s)=M(0)+\frac{M''(\sigma_s)}{2}s^2\leq \left(1+C\frac{T\varepsilon^2}{t_0^3}\right)M(0),
\end{align}
where  $\sigma_s\in(0,s)$. 
Finally, we recall that 
\begin{align*}
    M(0)=\frac{r_0^2}{1+r_0^2}, \;\mbox{ where } (1+r_0^2)^N=\frac{T-\varepsilon^2}{t-\varepsilon^2}.
\end{align*}
Then,
\begin{align*}
    M(0)=1-\left(\frac{t-\varepsilon^2}{T-\varepsilon^2}\right)^{1/N},
\end{align*}
with $\varepsilon^2<t_0/2$ by \eqref{eq:T0}.
Let $f(x)={1-}\left(\frac{t-x}{T-x}\right)^{1/N}$ for $x<t/2$. By the mean value theorem, for any $x<t/2$ there exists some $\tilde x\in(0,x)$ such that 
\begin{align*}
    \frac{f(x)-f(0)}x=f'(\tilde x)=\frac{1}{N}\frac{T-t}{(t-\tilde x){(T-\tilde x) }}\left(1-f(\tilde x)\right) .
\end{align*}
Since $f$ is a non-decreasing function in $(0, t/2)$ and $t-\tilde x>t/2$, the following estimate holds:
\begin{align*}
    \frac{f(x)-f(0)}x
    &\leq \frac{4}{N{T}}\left(\frac{T}{t}-1\right)\left(1-f(0)\right)
    =\frac{4}{N{T}}\left(\frac{T}{t}-1\right)\left(\frac{t}{T}\right)^{1/N}\\
    & \leq 4  \left(\left(\frac{T}{t}\right)^{1/N}-1\right)\frac{{1}}{t}\left(\frac{t}{T}\right)^{1/N}
    =2\frac{1}{t} f(0).
\end{align*}
Taking into account that  $t\geq t_0$, we finally infer
\begin{align*}
    f(x)\leq\left(1+\frac{4}{t_0}x\right)f(0).
\end{align*}
Since $M(0)=f(\varepsilon^2)$, we conclude
\begin{align*}
    M(0)
    &\leq \left(1+2\frac{\varepsilon^2}{t_0}\right)\left(1-\left(\frac{t}{T}\right)^{1/N}\right)
\end{align*}
Combining this with \eqref{eq:estimateMs} for $s=1$ and recalling \eqref{eq:epsilonT}, the result follows with $C_0=\frac{C}{t_0^3}$ and $T_0=\max\{\frac12, t_0, T_0'\}$, where $T_0'$ satisfies the equality in \eqref{eq:T0}.
\end{proof}

As we have previously announced, the following lemma moves away from the series of lemmas in \cite{GGRT}*{Section 2}, but contains  results that resemble some in \cite{NicRic} for the Fock space.

\begin{lemma}\label{lem:on_mu}
    Let $P\in\mathcal P_N$ with $\|P\|_{N}=1$ and $T<1$.
    Then, 
    \begin{align}\label{desigualtatdiferencial}
       \mu'(t)\leq -\frac{1}{Nt}\big(1-\mu(t)\big), \quad t\in(0,T).
    \end{align}
    In addition, the functions
    \begin{align*}
        &\frac{\mu(t)-\mu_0(t)}{t^{1/N}} \quad\mbox{ and }\quad\frac{\int_0^t \big(\mu(\tau)-\mu_0(\tau)\big)d\tau}{t^{1+1/N}}
    \end{align*}
    are non-increasing in $(0, T)$.
\end{lemma}

\begin{proof}
     Let $v=\frac{1}{2} \log u$ and $\nu(t)=m(\{v(z)>t\})$.
    We seek to apply Theorem 1.1 in \cite{KNOCT}. 
    Firstly, we notice that $\Delta_M=\pi(1+|z|^2)^2\Delta$, where $\Delta$ is the ordinary Euclidean Laplacian.
    Then, 
    \begin{align*}
        \Delta_M v={\pi}(1+|z|^2)^2\left(
        \Delta \left(\log |P(z)|\right)-\frac{N}{2}\Delta\left(\log\left(1+|z|^2\right)\right)\right).
    \end{align*}
    Since $\Delta =4\partial_{z}\partial_{\bar z} $, we achieve 
    \begin{align*}
        \Delta_M v 
        =-2\pi N.
    \end{align*}

    Secondly, $H(x)=4\pi x\left(1-x\right)$ by the isoperimetric inequality for the 2-dimensional sphere of radius $\frac{1}{2\sqrt{\pi}}$ (see e.g., \cite{Oss78}).
     Therefore, by \cite{KNOCT}*{Theorem 1.1}, 
    \begin{align}\label{eq:nu'}
        \nu'(t)\leq -2\frac{1}{N} \left(1-\nu(t)\right), \quad t\in\left(-\infty, \frac{1}{2}\log T\right).
    \end{align}
    Since $\nu(t)=\mu(e^{2t})$,  estimate \eqref{desigualtatdiferencial} follows. 

    Let $F(t)=\frac{\mu(t)-\mu_0(t)}{t^{1/N}}$. Then
    \begin{align*}
        F'(t)=\frac{1}{t^{1/N}}\left(\mu'(t)-\frac{1}{Nt}\mu(t)-\left(\mu_0'(t)-\frac{1}{Nt}\mu_0(t)\right)\right).
    \end{align*}
    Applying \eqref{desigualtatdiferencial}, which becomes an equality  for $\mu_0(t)$, we conclude that $F'(t)\leq 0$.  

    Finally, we have that 
    \begin{align*}
    \left(\frac{\int_0^t \big(\mu(\tau)-\mu_0(\tau)\big)d\tau}{t^{1+1/N}}\right)'
    =\frac{1}{t} G(t),
    \end{align*}
    where 
    \begin{align*}
        G(t)
        &=\frac{\mu(t)-\mu_0(t)}{t^{1/N}}-\left(1+\frac{1}{N}\right)\frac{\int_0^t \big(\mu(\tau)-\mu_0(\tau)\big)d\tau}{t^{1+1/N}}\\
        &=F(t)-\left(1+\frac{1}{N}\right)\frac{\int_0^t F(\tau)\tau^{1/N}d\tau}{t^{1+1/N}}.
    \end{align*}
    Taking into account that $F(t)$ is non-increasing, we infer 
    \begin{align*}
        G(t)
        &\leq F(t)-\left(1+\frac{1}{N}\right)\frac{F(t)\int_0^t \tau^{1/N}d\tau}{t^{1+1/N}}=0.
    \end{align*}
\end{proof}

\begin{lemma}\label{lem:unique_t*}
    Let $P\in\mathcal P_N$ with $\|P\|_{N}=1$ and $T<1$.
    Then there exists a unique $t^*\in(0, T)$ such that 
    \begin{align}\label{eq:t*}
    \begin{split}
        \mu(t)\geq \mu_0(t) & \mbox{ if } t\in(0,t^*],\\
        \mu(t)\leq \mu_0(t) & \mbox{ if } t\in[t^*, T].
        \end{split}
    \end{align}
    In addition there exists a universal  constant $T^*\in(0,1)$ (independent of $N$) such that $t^*\leq T^*$.
\end{lemma}

Notice that not only $\mu(t^*)=\mu_0(t^*)$, but  also $\mu(0)=\mu_0(0)=1$. In addition, in the proof it can be seen that a sharper upper bound for $t^*$ could be given if we allow it to depend on $N$. 

\begin{proof}[Proof of Lemma~\ref{lem:unique_t*}]
    By the second part of Lemma~\ref{lem:on_mu}, we know that  $\{t\in(0,1): \mu(t)=\mu_0(t)\}$ is a connected, non-empty interval. 
    We now prove that it has an empty interior, arguing by contradiction. 
    
    Let us assume that $\mu(t)=\mu_0(t)$ for $t\in(t_1, t_2)\subset(0, T)$ and $T<1$ and let $\nu(t)=\mu(e^{2t})$, which satisfies \eqref{eq:nu'}. 
    By hypothesis, $\nu(t)=(1-e^{2t/N})$ for $t\in I=(e^{2t_1}, e^{2t_2})$, so the equality is attained in  \eqref{eq:nu'} for $t\in I$. 
    Let $A_t=\{v(z)>t\}$. From the proof of \cite{KNOCT}*{Theorem~1.1}, we conclude that the equality is achieved provided $A_t$ is a disc, $|\nabla v|$ is constant on $\partial A_t$ and $\Delta_M v(z) = -2\pi N$ if $z\in A_t$, for $t\in I$.

    We firstly see that  $A_t$ are concentric discs for $t\in I$. Assume $A_t=\mathbb D_{r_t}(z_t)$.
    Since $v(z_t+\omega r_t)=t$ and $\nabla v(z_t+\omega r_t) = c_t \omega$ for all $\omega\in\mathbb S^1$, we have
    \[1=\nabla v(z_t+\omega r_t) \cdot \partial_t(z_t+\omega  r_t)= c_t(\partial_t z_t \cdot \omega +\partial r_t).\]
    Since this must hold for any $\omega$,  we conclude that $\partial_t z_t=0$, i.e., $z_t$ does not depend on $t$. 
    We may assume without loss of generality that $z_t = 0$.

    Now, we note that  $\Delta_M v = -2\pi N$ in $A_t$ implies that $v + \frac{N}{2}\log(1+|z|^2)$ is harmonic and radial. Since it is constant  on $\partial A_t$, it is constant in $A_t$. This means that $u$ must be of the form $u(z) = \frac 1{(1+|z|^2)^N}$, and this contradicts that $T<1$.

    In order to see the existence of an universal upper bound for $t^*$, let us apply Lemma~\ref{lemma:super-level} with $t_0=\frac 12$.
    If $t^*\geq t_0=\frac 12$ and $T\geq T_0$, where $T_0$ is the threshold for $t_0=\frac12$,  we can apply  \eqref{newestmu} at $t=t^*$ with $C_0=C_0(\frac12)$. Since $\mu(t^*)=\mu_0(t^*)$ we have
    \begin{align*}
        1-(t^*)^{1/N}\leq \big(1+C_0(1-T)\big)\left(1-\left(\frac{t^*}{T}\right)^{1/N}\right).
    \end{align*}
    This implies
    \begin{align*}
        t^*\leq \left(\frac{C_0T^{1/N}(1-T)}{1-T^{1/N}+C_0(1-T)}\right)^N.
    \end{align*}
    Taking into account that the right-hand side is an increasing function of $T$ in $(0,1)$, we  infer
     \begin{align*}
        t^*\leq \left(\frac{C_0N}{1+C_0N}\right)^N.
    \end{align*}
    Finally, we notice that the right-hand side decreases with $N$, so
    \begin{align*}
        t^*\leq \frac{C_0}{1+C_0}.
    \end{align*}

    If $t^*<t_0=\frac 12$ or $T<T_0$, then it also holds that 
    \begin{align*}
        t^*\leq T^*:=\max\left\{\frac 12, T_0, \frac{C_0}{1+C_0}\right\}<1.
    \end{align*}

\end{proof}

\begin{lemma}\label{lem:upper}
    Let $P\in\mathcal P_N$ with $\|P\|_{N}=1$. For every $t_0\in(0,1)$ there holds
    \begin{align}\label{eq:upper}
        \int_{t^*}^1\left(\mu_0(t)-\mu(t)\right)dt
        \leq 
      \left(1-\mu(t_0)\right)^{-(N+1)}\left(\frac{1}{N+1}-\frac{\int_{\{u(z)>t_0\}}u(z)dm(z)}{1-\left(1-\mu(t_0)\right)^{N+1}}\right),
    \end{align}
    where $t^*$ is the unique value in $(0,T)$ satisfying \eqref{eq:t*}.
\end{lemma}

\begin{proof}
     Let $s^*=\mu(t^*)=\mu_0(t^*)$. Then 
    \begin{align*}
        I&= \int_{t^*}^1\left(\mu_0(t)-\mu(t)\right)dt
        =\int_{t^*}^1\left(\mu'(t)-\mu_0'(t)\right)tdt
         \\&=\int_{0}^{s^*}\left(\mu_0^{-1}(s)-\mu^{-1}(s)\right)ds,
    \end{align*}
    where for the last identity we have applied suitable changes of variables. Recall that 
    $$\mu_0^{-1}(s)=\left(1-s\right)^N.$$

    Firstly, we observe that  the function 
    \begin{align*}
        \rho(s)=\frac{\mu^{-1}(s)}{\mu_0^{-1}(s)}, \;\; s\in[0,1).
    \end{align*}
    is non-decreasing. 
    Indeed, the sign of $\rho'(s)$ is determined by the sign of 
    \begin{align*}
        (\mu^{-1})'(s)-\frac{(\mu_0^{-1})'(s)}{\mu_0^{-1}(s)}\mu^{-1}(s)
        =\frac{1}{\mu'\big(\mu^{-1}(s)\big)}+\frac{N}{1-s}\mu^{-1}(s).
    \end{align*}
    Taking $t=\mu^{-1}(s)$, $\rho'(s)\geq0$ for $s\in[0,1)$ if and only if
    \begin{align*}
       \frac{1}{\mu'(t)}+\frac{Nt}{1-\mu(t)}\geq0 \;\mbox{ for } t\in(0,T),
    \end{align*}
    which holds by Lemma~\ref{lem:on_mu}.

    For any $0\leq s_1 <s_2\leq 1$, let $I(s_1, s_2)=\int_{s_1}^{s_2}\left(\mu_0^{-1}(s)-\mu^{-1}(s)\right)ds$.
    With this notation, \eqref{eq:upper} is equivalent to
    \begin{align*}
        I\leq \frac{\left(1-\mu(t_0)\right)^{-(N+1)}}{N+1}\eta(\mu, t_0)
    \end{align*}
    where 
    \begin{align*}
        I&=I(0, s^*)=-I(s^*, 1),\\
        \eta(\mu, t_0)&=\frac{I(0, s_0)}{\int_0^{s_0} \mu_0^{-1}(s)ds}, \quad s_0=\mu(t_0).
    \end{align*}

    \noindent\emph{Case 1: $t_0\geq t^*$, i.e. $s_0\leq s^*$.}
    Taking into account the monotonicity of $\rho$, we obtain
    \begin{align*}
        I(0, s_0)&\geq \big(1-\rho(s_0)\big) \int_{0}^{s_0} \mu_0^{-1}(s)ds,
        \\
        I(s_0, s^*)&\leq \big(1-\rho(s_0)\big)\int_{s_0}^{s^*}\mu_0^{-1}(s)dt.
    \end{align*}
    Combining the previous inequalities, we infer
    \begin{align*}
        I&=I(0, s^*)=I(0, s_0)+I(s_0, s^*)
        \leq \left(1+\frac{\int_{s_0}^{s^*}\mu_0^{-1}(s)ds}{\int_0^{s_0}\mu_0^{-1}(s)ds}\right) I(0, s_0)\\
        &\leq \frac{\int_{0}^{s^*}\mu_0^{-1}(s)ds}{\int_0^{s_0}\mu_0^{-1}(s)ds} I(0, s_0)
        \leq \frac{\frac{1}{N+1}}{\int_0^{s_0}\mu_0^{-1}(s)ds} I(0, s_0)\leq\frac{1}{N+1}\eta(\mu, t_0).
    \end{align*}

    \emph{Case 2: $t_0\leq t^*$, i.e. $s_0\geq s^*$.} 
    Arguing as in the previous case, we have 
    \begin{align*}
        I(s^*, s_0)&\geq \Big(\frac{1}{\rho(s_0)}-1\Big) \int_{s^*}^{s_0} \mu^{-1}(s)ds,
        \\
        I(s_0,1)&\leq \Big(\frac{1}{\rho(s_0)}-1\Big) \int_{s_0}^{1} \mu^{-1}(s)ds.
    \end{align*}
    and hence, since $\rho(s^*)=1$,
    \begin{align*}
        I&=-I(s^*, 1)=-I(s^*, s_0)-I(s_0, 1)
        \leq - \left(1+\frac{\int_{s^*}^{s_0}\mu^{-1}(s)ds}{\int_{s_0}^{1}\mu^{-1}(s)ds}\right) I(s_0, 1)\\
        &\leq \frac{\int_{s^*}^{1}\mu^{-1}(s)ds}{\int_{s_0}^{1}\mu^{-1}(s)ds}I(0, s_0)
        \leq  \frac{\frac{1}{N+1}}{\int_{s_0}^{1}\mu_0^{-1}(s)ds}I(0, s_0)= 
        \frac{C(s_0)}{N+1}\eta(\mu, t_0),
    \end{align*}
    where 
    \begin{align*}
         C(s_0)&=\frac{\int_0^{s_0}\mu_0^{-1}(s)ds}{\int_{s_0}^{1}\mu_0^{-1}(s)ds}
        =\left(1-s_0\right)^{-(N+1)}-1\leq \left(1-s_0\right)^{-(N+1)}.
    \end{align*}
    
\end{proof}

\begin{lemma}\label{lem:lower}
    There exists a constant $C\in(0,1)$   such that for any  $P\in\mathcal P_N$ with $\|P\|_{N}=1$ it holds
    \begin{align}
        \int_{t^*}^1\left(\mu_0(t)-\mu(t)\right)dt
        \geq \frac{C}{N}(1-T),
    \end{align}
    where $t^*$ is the unique value in $(0,T)$ satisfying \eqref{eq:t*}.
\end{lemma}

\begin{proof}
    Let $T^*$ be the universal bound for $t^*$ in Lemma~\ref{lem:unique_t*}.
    Then, provided $T>T^*$,
    \begin{align*}
        I=\int_{t^*}^1\left(\mu_0(t)-\mu(t)\right)dt\geq \int_{T^*}^T\left(\mu_0(t)-\mu(t)\right)dt.
    \end{align*}
    Now we apply Lemma~\ref{lemma:super-level} with $t_0=T^*$. 
    Then there exists $T_0\geq T^*$ such that if $T\geq T_0$, for $t\in[T^*, T]$ it holds
    \begin{align*}
        I\geq \int_{T^*}^T g(t)dt,
    \end{align*}
    where 
    \begin{align*}
        g(t)
        &=1-t^{1/N}-\big(1+C_0(1-T)\big)\left(1-\left(\frac{t}{T}\right)^{1/N}\right)\\
        &=\left(1-T^{1/N}\right)\left(\frac{t}{T}\right)^{1/N}-C_0(1-T)\left(1-\left(\frac{t}{T}\right)^{1/N}\right).
    \end{align*}
    Using that $(N+1)(1-T^{1/N})\geq {1-T}$ and $T<1$,
    \begin{align*}
        g(t)
        &\geq \frac{1}{N+1}(1-T)\left(t^{1/N}-C_0(N+1)\left(1-t^{1/N}\right)\right)=\frac{1-T}{N+1}\gamma(t).
    \end{align*}
    Notice that $\gamma(t)$ is an increasing function with $\gamma(1)=1$. Let  $T_1\in (T_0, 1)$ be sufficiently close to $1$ so $\gamma(T_1)=c_1>0$.
    Finally, let us fix  $T_2\in(T_0,T_1)$ large.
    If $T>T_2$,
    \begin{align*}
        I\geq \int_{T_1}^{T_2} g(t)dt\geq \frac{1-T}{N+1} \gamma(T_1)(T_2-T_1)
        =\frac{c_1(T_2-T_1)}{N+1}(1-T). 
    \end{align*}

    If $T\leq T_2$ (which in particular includes the case $T<T^*$), we  exploit that $t^*<T$ to conclude 
    \begin{align*}
    \begin{split}
    I&\geq \int_{T}^1 \big(\mu_0(t)-\mu(t)\big)dt
        =\int_T^1\left(1-t^{1/N}\right)dt
        \\&=(1-T)-\frac{N}{N+1}(1-T^{1+1/N})\geq\frac{(1-T)^2}{2N}\geq \frac{1-T_2}{2N}(1-T).
        \end{split}
    \end{align*}
    Therefore, taking $C=\min\left\{c_1(T_2-T_1), \frac{1-T_2}{2} \right\}$, the desired estimate holds.

\end{proof}

\section{Proof of the main results}
\label{sec:proofs}

In this section, we present the proofs of the results in Section~\ref{sec:intro}. In addition to the lemmas in Section~\ref{sec:superlevelsets} we need the following technicality:

\begin{lemma}\label{lem:valuemin}
    Let $P\in\mathcal{P}_N$ with $\|P\|_{N}=1$.
    Then 
    \begin{align*}
        D_N(P)^2=2(1-\sqrt{T}).
    \end{align*}
\end{lemma}

\begin{proof}
For any $a\in \C$ and $\theta\in\mathbb{S}^1$
    $$\|P-e^{i\theta}\kappa_a\|_{N}^2 = 2 (1-\Re \langle \kappa_a,
e^{-i\theta}P\rangle_{N})= 2 \left(1-\frac{\Re P(a) e^{-i\theta}}{(1+|a|^2)^{N/2}}\right).$$

Optimizing in $\theta$ one gets
\begin{align*}
    \min\{\|P-e^{i\theta}\kappa_a\|_{N}^2, \theta\in \mathbb{S}^1\}=2 \left(1-\frac{|P(a)|}{(1+|a|^2)^{N/2}}\right)
    =2\left(1-\sqrt{u(a)}\right).
\end{align*}
Since $\sup_{z\in\C}u(z)=T$, the result follows. 

\end{proof}

\begin{proof}[Proof of Theorem~\ref{thm:concentration}]
    Let $t_0\in(0,1)$ be such that  $\mu(t_0)=m(\Omega)$. Combining Lemmas~\ref{lem:valuemin},~\ref{lem:lower} and~\ref{lem:upper} we obtain 
    \begin{align*}
        D_N(P)^2\leq C\left(1-m(\Omega)\right)^{-(N+1)}\left(1-\frac{N+1}{1-\left(1-m(\Omega)\right)^{N+1}}\int_{\{u(z)>t_0\}} u(z) dm(z)\right).
    \end{align*}
    Finally, since  $m(\Omega)=m\big(\{u(z)>t_0\}\big)$, 
    \begin{align*}
        \int_{\{u(z)>t_0\}} u(z) dz\geq \int_{\Omega} u(z) dz
    \end{align*}
    and the claim follows.
\end{proof}

\begin{proof}[Proof of Proposition~\ref{prop:domains}]
    Throughout this proof, we skip the subindex $N$ in $\delta_N$ and $m$ in $\mathcal A_m$ and we define $K_\Omega=\left(1-m(\Omega)\right)^{-(N+1)}>1$.
    We seek then to prove that 
    \begin{align}\label{eq:goal}
        \mathcal A(\Omega)\leq C\frac{K_\Omega^{3/2}}{m(\Omega)} \delta(P, \Omega)^{1/2}.
    \end{align}
    
    Arguing as in Step 1 of the proof of Lemma~\ref{lemma:super-level},we assume
    \begin{align*}
        P(z)=\sqrt{T}+\varepsilon Q(z),
    \end{align*}
    with $\|Q\|_{N}=1$ and $\varepsilon^2=1-T$.
    By Lemmas~\ref{lem:upper} and~\ref{lem:lower}, 
    \begin{align}\label{eq:espilon}
        \varepsilon
        &\leq C_0 \big(K_\Omega\delta(P,\Omega)\big)^{1/2}.
    \end{align}
   
    If $\varepsilon\geq \frac{1}{50}\left(1-m(\Omega)\right)^{N+1}=\frac{1}{50 K_\Omega}$, applying  \eqref{eq:espilon} we directly achieve \eqref{eq:goal} since
    \begin{align*}
        \mathcal A(\Omega)
        &\leq \frac{1}{m(\Omega)}  
        \leq \frac{C}{m(\Omega)} K_\Omega\ \varepsilon \leq C K_\Omega^{3/2}\delta(P, \Omega)^{1/2}.
    \end{align*}

    From now on, we focus on the case  $\varepsilon< \frac{1}{50}\left(1-m(\Omega)\right)^{N+1}$,
    which in particular satisfies that 
    \begin{align}\label{eq:epsilonsmall}
       (1-\varepsilon^2) \left(1-m(\Omega)\right)^N-4\varepsilon>\frac{1}{2}\left(1-m(\Omega)\right)^N.
    \end{align}
    We divide the proof into different steps.

    \noindent\emph{Step 1: General considerations.}
    Let $t_\Omega\in(0,T)$ be such that $\mu(t_\Omega)=m(\Omega)$.
    We denote by $A_\Omega=\{z\in\C: u(z)\geq t_\Omega\}$ and  define
    \begin{align*}
        d(\Omega)=\int_{A_\Omega}u(z)dm(z)-\int_\Omega u(z)dm(z).
    \end{align*}
    By the qualitative inequality in Theorem~\ref{thm:concentration}
    \begin{align*}
        d(\Omega)
        \leq \int_{\Omega^*} u_0(z)dm(z)-\int_\Omega u(z)dm(z)
        =\frac{1-K_\Omega^{-1}}{N+1} \delta (P,\Omega).
    \end{align*}

    Moreover, for any $z\in\C$ it holds
    \begin{align}\label{eq:u-Tu0}
        Tu_0(z)-u(z)
        &=\frac{T-|P(z)|^2}{(1+|z|^2)^N}
        \leq \frac{|\sqrt{T}-P(z)|}{(1+|z|^2)^{N/2}}\frac{\sqrt{T}+|P(z)|}{(1+|z|^2)^{N/2}}
        \leq 2\sqrt{T}\varepsilon\leq 2 \varepsilon.
    \end{align}
    Hence,
    \begin{align*}
        \left\{z\in\C: u_0(z)\geq \frac{t_\Omega+2\varepsilon}{T}\right\}\subset A_\Omega\subset \left\{z\in\C: u_0(z)\geq \frac{t_\Omega-2\varepsilon}{T}\right\}.
    \end{align*}
    This implies
    \begin{align*}
    \mu_0\left(\frac{t_\Omega+2\varepsilon}{T}\right)\leq m(\Omega)\leq \mu_0\left(\frac{t_\Omega-2\varepsilon}{T}\right)
    \end{align*}
    and therefore,
    \begin{align*}
        T\left(1-m(\Omega) \right)^N-2\varepsilon\leq t_\Omega\leq T\left(1-m(\Omega)\right)^N+2\varepsilon.
    \end{align*}
    In particular, it follows that 
    \begin{align}\label{eq:AOmegadiscs}
        \left\{z\in\C: u_0(z)\geq  t_0 +\frac{4\varepsilon}{T}\right\}\subset A_\Omega\subset \left\{z\in\C: u_0(z)\geq t_0-\frac{4\varepsilon}{T}\right\}.
    \end{align}
    where $t_0=\left(1-m(\Omega) \right)^N$.

    \noindent\emph{Step 2: The set $B$.}
    Let $\mathcal T: A_\Omega\backslash\Omega\to\Omega\backslash A_\Omega$ be a transport map and let 
    \begin{align*}
        B=\{z\in A_\Omega\backslash\Omega: |\mathcal T(z)|^2>|z|^2+C_\Omega\gamma\},
    \end{align*}
    where $C_\Omega$ and $\gamma$ will be chosen later to ensure that
    \begin{align}\label{eq:uinB}
        u(z)-u(\mathcal T(z))\geq \gamma \quad \forall z\in B.
    \end{align}
    Combining this with the direct estimate
    \begin{align*}
        \int_B\big(u(z)-u(\mathcal T(z))\big)dm(z)\leq d(\Omega),
    \end{align*}
    one has that 
    \begin{align}\label{eq:m(B)}
        m(B)\leq \frac{d(\Omega)}{\gamma}\leq  \frac{\left(1-K_\Omega^{-1}\right)}{(N+1)\gamma} \delta (P,\Omega).
    \end{align}

    In order to obtain \eqref{eq:uinB}, we start observing that  \eqref{eq:u-Tu0} implies
    \begin{align}\label{eq:uinBpre}
    \begin{split}
        u(z)-u(\mathcal T(z))
        &\geq T|u_0(z)-u_0(\mathcal T(z))|-4\varepsilon
        \\
        &\geq Tu_0(z)\left(1-\left(\frac{1+|z|^2}{1+|\mathcal T(z)|^2}\right)^N\right)-4\varepsilon.
        \end{split}
    \end{align}
    
    On the one hand, by \eqref{eq:AOmegadiscs}, for any $z\in A_\Omega$, 
    \begin{align*}
        u_0(z)\geq t_0-\frac{4\varepsilon}{T}.
    \end{align*}
    Recalling that $T=1-\varepsilon^2$ and $t_0=\left(1-m(\Omega) \right)^N$, we use \eqref{eq:epsilonsmall} to finally infer
    \begin{align}\label{eq:Tu0bis}
        Tu_0(z)\geq \frac{1}{2}t_0.
    \end{align}

    On the other hand,  if  $z\in B$, it holds that 
    \begin{align*}
        1-\left(\frac{1+|z|^2}{1+|\mathcal T(z)|^2}\right)^N\geq 1-\frac{1+|z|^2}{1+|\mathcal T(z)|^2}=\frac{|\mathcal T(z)|^2-|z|^2}{1+|z|^2+\big(\mathcal T(z)|^2-|z|^2\big)}.
    \end{align*}
    If $|\mathcal T(z)|^2-|z|^2>1$, noticing that $\lambda\mapsto\frac{\lambda}{1+|z|^2+\lambda}$ is  increasing for $\lambda>0$, we have 
    \begin{align*}
        1-\left(\frac{1+|z|^2}{1+|\mathcal T(z)|^2}\right)^N
        \geq \frac{1}{2+|z|^2}\geq \frac{1}{2}\left(u_0(z)\right)^{1/N}.
    \end{align*}
    If on the contrary $|\mathcal T(z)|^2-|z|^2\leq 1$ (and  $|\mathcal T(z)|^2-|z|^2\geq C_\Omega \gamma$),
    \begin{align*}
        1-\left(\frac{1+|z|^2}{1+|\mathcal T(z)|^2}\right)^N
        \geq \frac{C_\Omega \gamma}{2+|z|^2}\geq \frac{C_\Omega \gamma}{2}\left(u_0(z)\right)^{1/N}.
    \end{align*}

    Combining \eqref{eq:uinBpre} with \eqref{eq:Tu0bis} and the last two estimates one gets
    \begin{align*}
        u(z)-u(\mathcal T(z))
        &\geq \frac{1}{8} \min\{1,C_\Omega\gamma\}t_0^{1+1/N}-4\varepsilon.
    \end{align*}
    Choosing $C_\Omega=40 t_0^{-(1+1/N)}$ and $\varepsilon\leq\gamma\leq C_0\big(K_\Omega \delta(P,\Omega)\big)^{1/2}$, \eqref{eq:uinB} holds provided $$C_\Omega\gamma\leq 40C_0K_\Omega^{3/2}\delta(P,\Omega)^{1/2}<1.$$
    Otherwise, the estimate \eqref{eq:goal} is immediate.

   \noindent \emph{Step 3: Estimate for $\mathcal A(\Omega, A_\Omega)$.}
   Notice that 
    \begin{align}\label{eq:AOmAOm}
        \mathcal A(\Omega, A_\Omega)=\frac{2 m(A_\Omega\backslash\Omega)}{m(\Omega)}
        =\frac{2}{m(\Omega)}\Big(m(B)+m\big((\Omega\backslash\mathcal T(B))\backslash A_\Omega\big)\Big).
    \end{align}
    In view of \eqref{eq:m(B)}, it remains estimating $m\big((\Omega\backslash\mathcal T(B))\backslash A_\Omega\big)$.
    
    The inclusions in  \eqref{eq:AOmegadiscs} together with 
    \begin{align*}
        \Omega\backslash \mathcal T(B)
        \subset\left\{ z\in\Omega\backslash A_\Omega: |z|^2\leq C_\Omega \gamma +|w|^2 \mbox { for some } w\in A_\Omega\right\}
    \end{align*}
    yield to 
    \begin{align*}
        (\Omega\backslash\mathcal T(B))\backslash A_\Omega
        &\subset \mathbb D_{R_1}\backslash \mathbb D_{R_2},
    \end{align*}
    where 
    \begin{align*}
        R_1^2 &= C_\Omega\gamma+\left(t_0-\frac{4\varepsilon}{T}\right)^{-1/N}-1,\\
        R_2^2 &=\left(t_0+\frac{4\varepsilon}{T}\right)^{-1/N}-1
    \end{align*}
    with $t_0=\left(1-m(\Omega) \right)^N$.
    Then, 
    \begin{align*}
        m\big((\Omega\backslash\mathcal T(B))\backslash A_\Omega\big)
        &\leq m(\mathbb D_{R_1})-m(\mathbb D_{R_2})=\frac{1}{1+R_2^2}-\frac{1}{1+R_1^2}\leq R_1^2-R_2^2.
    \end{align*}
    Taking into account that $\varepsilon<\frac{t_0}{50}<\frac{1}{50}$, so $T>\frac{1}{2}$, we have
    \begin{align*}
        m\big(\Omega\backslash\mathcal T(B))\backslash A_\Omega\big)
        &\leq C_\Omega\gamma +\frac{C}{N}t_0^{-(1+1/N)}\frac{\varepsilon}{T}
        \\
        &\leq C \left(C_\Omega\gamma + K_\Omega \varepsilon\right).
    \end{align*}
    Recalling that   $C_\Omega=CK_\Omega$ and $\varepsilon\leq\gamma$, it suffices to take $\gamma=C_0\big(K_\Omega \delta(P,\Omega)\big)^{1/2}$ to conclude from \eqref{eq:m(B)} and \eqref{eq:AOmAOm} that
    \begin{align}\label{eq:AOmAOmfinal}
    \begin{split}
        \mathcal A (\Omega, A_\Omega)
        &\leq \frac{C}{m(\Omega)} \left(
        \frac{K_\Omega-1}{K_\Omega^{3/2}}+K_\Omega^{3/2}\right)\delta(P,\Omega)^{1/2}
        \\&\leq C\frac{K_\Omega^{3/2}}{m(\Omega)}\delta(P,\Omega)^{1/2}.
    \end{split}
    \end{align}

    \noindent\emph{Step 4: Estimate for $\mathcal A(\Omega)$.}
    Let  $\Omega^*$ be the disc centered at the origin with $m(\Omega^*)=m(\Omega)$.
    Actually, $\Omega^*=\{z\in\C: u_0\geq t_0\}=\mathbb D_R$ with $R^2=t_0^{-1/N}-1$.
    Taking into account that
    \begin{align*}
        \Omega\backslash\Omega^*
        &\subset (\Omega\backslash A_\Omega)\cup(A_\Omega\backslash\Omega^*)
    \end{align*}
    we easily observe that
    \begin{align*}
        \mathcal A(\Omega)\leq \mathcal A(\Omega, \Omega^*)\leq \mathcal A (\Omega, A_\Omega)+\mathcal A (A_\Omega,\Omega^*).
    \end{align*}
    Recalling \eqref{eq:AOmegadiscs} and arguing as in Step 3, it follows that
    \begin{align*}
        \mathcal A (A_\Omega,\Omega^*)
        \leq \frac{1}{m(\Omega)}\left(\left(t_0-\frac{4\varepsilon}{T}\right)^{-1/N}-\left(t_0+\frac{4\varepsilon}{T}\right)^{-1/N}\right)\leq C \frac{K_\Omega^{3/2}}{m(\Omega)}\delta(P,\Omega)^{1/2}.
    \end{align*}
    This together with \eqref{eq:AOmAOmfinal} yields to the desired result \eqref{eq:goal}. 
    
\end{proof} 

\begin{proof}[Proof of Theorem~\ref{thm:Wehrl}]
    We follow the first proof of  \cite{FNT23}*{Theorem 3}. 
    Since $\Phi$ is convex, i.e. $\Phi'$ is non-decreasing, but non-linear, there exists $t_1, t_2$ such that $0<t_1<t_2<1$ and $\Phi'(t_1)<\Phi'(t_2)$. 
    
    Given $P\in\mathcal P_N$ with $\|P\|_{N}=1$ and let $t^*\in(0,T)$ be as in Lemma~\ref{lem:unique_t*}. Then 
    \begin{align*}
        \mathcal{S}&=\frac{S_\Phi(P)-S_\Phi(1)}{N+1}
        =\int_0^1 \Phi'(t)\big(\mu_0(t)-\mu(t)\big) dt\\
        &=\int_{0}^1\big(\Phi'(t)-\Phi'(t^*)\big)\big({\mu_0(t)-\mu(t)}\big) dt.
    \end{align*}
    If $t_1<t_2\leq t^*$, using the monotonicity of $\Phi'$, we have
    \begin{align*}
        \mathcal{S}&
        \geq \int_{0}^{t_1}\big(\Phi'(t)-\Phi'(t^*)\big)\big(\mu_0(t)-\mu(t)\big) dt\\
        &\geq \big(\Phi'(t_2)-\Phi'(t_1)\big)\int_{0}^{t_1}\big(\mu(t)-\mu_0(t)\big) dt.
    \end{align*}
    Similarly, if $t^*\leq t_1<t_2$, 
    \begin{align*}
        \mathcal{S}
        &\geq \big(\Phi'(t_2)-\Phi'(t_1)\big)\int_{t_2}^1\big(\mu_0(t)-\mu(t)\big) dt.
    \end{align*}
    Finally, if $t_1<t^*<t_2$, 
    \begin{align*}
        \mathcal{S}
        \geq &\big(\Phi'(t_2)-\Phi'(t^*)\big)\int_{t_2}^1\big(\mu_0(t)-\mu(t)\big) dt \\
        &\; +\big(\Phi'(t^*)-\Phi'(t_1)\big)\int_0^{t_1}\big(\mu(t)-\mu_0(t)\big) dt.
    \end{align*}

    Therefore, it remains to prove that for some fixed $t_0$, there exists a constant $C$ such that 
    \begin{align*}
        \int_0^{t_0}\big(\mu(t)-\mu_0(t)\big) dt&\geq C  \frac{D_N(P)^2}{N+1}\quad \mbox{ if } t_0<t^*,\\
        \int_{t_0}^1\big(\mu_0(t)-\mu(t)\big) dt&\geq C  \frac{D_N(P)^2}{N+1} \quad\mbox{ if } t_0> t^*.
    \end{align*}

    If $t_0<t^*<1$, using the last monotonicity result in Lemma~\ref{lem:on_mu}, we have
    \begin{align*}
          \int_0^{t_0}\big(\mu(t)-\mu_0(t)\big) dt
         \geq \left(\frac{t_0}{t^*}\right)^{1+1/N}\int_0^{t^*}\big(\mu(t)-\mu_0(t)\big) dt
         \geq t_0^{1+1/N} \int_{t^*}^{1}\big(\mu_0(t)-\mu(t)\big)\, dt.
    \end{align*}
    Applying  Lemma~\ref{lem:lower}, one finally obtains 
    \begin{align*}
        \int_{0}^{t_0}\big(\mu(t)-\mu_0(t)\big) \, dt\geq \frac{C}{N} t_0^{2}   (1-T).
    \end{align*}

    If $t_0\geq t^*$ we distinguish three cases depending on the position of $T$ with respect to $t_0$:

    \noindent
    \emph{Case 1: $t_0\geq T$.} Since $\mu(t)=0$ for $t>T$, 
     \begin{align*}
        \int_{t_0}^1 \big(\mu_0(t)-\mu(t)\big) \, dt
        &= \int_{t_0}^1 (1-t^{1/N}) \, dt
        =1-t_0-\frac{N}{N+1}(1-t_0^{1+1/N})\\
        &\geq \frac{(1-t_0)^2}{2N}.
    \end{align*}

    \noindent
    \emph{Case 2: $2(1-T) \ge 1-t_0$ and $t_0\leq T$.} Arguing as above,
    \begin{align*}
        \int_{t_0}^1 \big(\mu_0(t)-\mu(t)\big) \, dt
        &\geq \int_{T}^1 \big(\mu_0(t)-\mu(t)\big) \, dt= \int_{T}^1 (1-t^{1/N}) \, dt\\
        &\geq \frac{(1-T)^2}{2N}\geq \frac{1-t_0}{4N}(1-T).
    \end{align*}

    \noindent
    \emph{Case 3: $2(1-T) \le 1-t_0$.}
    This estimate directly implies that $t_0>T$.
    We notice that the  function $H(s) =  \int_s^1 \bigl(\mu_0(t)-\mu(t)\bigr)\, dt$ is concave in $(t^*,T)$ because we know that $\mu_0' -\mu'>0$ in that interval. 
    Thus, $\frac{H(t^*)-H(t_2)}{t_2-t^*} \le \frac{H(t_2)-H(T)}{T-t_2}$. Therefore,
    \[
    \int_{t^*}^1\bigl(\mu_0(t)-\mu(t)\bigr)\, dt \le \frac{T-t^*}{T-t_0} \int_{t_0}^1\bigl(\mu_0(t)-\mu(t)\bigr)\, dt
    \le \frac{2}{1-t_0} \int_{t_0}^1\bigl(\mu_0(t)-\mu(t)\bigr)\, dt.
    \]
    Applying  Lemma~\ref{lem:lower}, we finally obtain 
    \begin{align*}
        \int_{t_0}^1\big(\mu_0(t)-\mu(t)\big) \, dt\geq C \frac{1-t_0}{2N} (1-T).
    \end{align*}

Combining all the previous estimates with  Lemma ~\ref{lem:valuemin}, which implies
\begin{align*}
    D_N(P)^2\leq 2(1-T)\leq 1,
\end{align*}
the result follows. 
\end{proof}

\section{Quantitative estimates in the Bargmann-Fock space}
\label{sec:Fock}

Let $\mathcal F^2$ be the Bargmann-Fock space of entire functions $f(z)$  with
\begin{align*}
    \|f\|_{\mathcal F^2}^2=\int_\C|f(z)|^2 e^{-\pi |z|^2}dz<\infty
\end{align*}
and define the  Hermitian product 
\begin{align*}
    \langle f,g\rangle_{\mathcal F^2}=\int_\C f(z)\overline{g(z)}e^{-\pi z^2} dz,
\end{align*}
where $dz=dx dy$ for $z=x+iy\in\C$. The Bargmann-Fock space
$\mathcal F^2$ endowed with the previous product is a reproducing kernel Hilbert space with kernel $k(z,\zeta)=e^{z\bar\zeta}$.

In this section, we obtain  quantitative estimates for the concentration inequality and for a generalized Wehrl entropy bound as the limit when $N$
goes to infinity of Theorems~\ref{thm:concentration} and~\ref{thm:Wehrl}. 
In this way, we recover the results in \cite{GGRT} and \cite{FNT23}, respectively.

We start studying the limit as $N\to \infty$ of the quantities involve in the main results of Section~\ref{sec:intro}.
Given any polynomial $P$, we define the following rescaling
\begin{align*}
    P^N(z)=P \left(\sqrt{\frac{N}{\pi}}z\right).
\end{align*}

\begin{lemma}\label{lem:limFock1}
Let $P, Q\in\mathcal P_M$. 
Then 
\begin{align*}
    \lim_{N\to \infty} \langle P^N,Q^N\rangle_{N}=\langle P, Q\rangle_{\mathcal F^2},
\end{align*}
so in particular $\lim_{N\to\infty}\|P^N\|_{N}=\|P\|_{\mathcal F^2}$.
\end{lemma}

\begin{proof}
    For any $N\geq M$ and performing a suitable change of variables we have
    \begin{align*}
        \langle P^N,Q^N\rangle_{N}
        &=(N+1)\int_\C\frac{P^N(z)\overline{Q^N(z)}}{\pi(1+|z|^2)^{N+2}}dz\\
        &=\frac{N+1}{N}\int_\C \frac{P(z)\overline{Q(z)}}{\left(1+\frac{\pi|z|^2}{N}\right)^{N+2}}dz. 
    \end{align*}
    Therefore, by the dominated convergence theorem,
    \begin{align*}
        \lim_{N\to\infty}\langle P^N,Q^N\rangle_{N}
        &=\int_\C \lim_{N\to\infty}\frac{N+1}{N} \frac{P(z)\overline{Q(z)}}{\left(1+\frac{\pi|z|^2}{N}\right)^{N+2}}dz
        \\&=\int_\C P(z)\overline{Q(z)}e^{-\pi|z|^2}dz=\langle P, Q\rangle_{\mathcal F^2}.
    \end{align*}
\end{proof}

For any set $\Omega\subset\C$, we define the following rescaled sets 
\begin{align*}
    \Omega^N=\sqrt{\frac{\pi}{N}}\Omega.
\end{align*}
Moreover, $\mathcal A(\Omega)$ denotes the Fraenkel asymmetry associated with the Lebesgue measure given by
\begin{align*}
    \mathcal A(\Omega)=\inf\left\{\frac{2|\Omega\backslash\mathbb D_\rho(z)|}{|\Omega|}: |\mathbb D_\rho(z)|=|\Omega|, z\in\C\right\}.
\end{align*}

\begin{lemma}\label{lem:limFock2}
    Let $\Omega\subset\C$ be a measurable set of finite Lebesgue measure. Then 
    \begin{align*}
         \lim_{N\to\infty}\big(1-m(\Omega^N)\big)^{N+1}&= e^{-|\Omega|},\\
         \lim_{N\to\infty} \mathcal A_m (\Omega^N)=\mathcal A(\Omega).
    \end{align*}
\end{lemma}
\begin{proof}
    Let $R>0$ and denote by $\Omega_R = \Omega\cap \mathbb D_R$. 
   Since 
   \begin{align*}
       m(\Omega^N_R)=\frac{1}{N}\int_{\Omega_R} \left(1+\frac{\pi|z|^2}{N}\right)^{-2} dz,
   \end{align*}
   it holds
   \begin{align*}
       \left(1+\frac{\pi R^2}{N}\right)^{-2}\frac{|\Omega_R|}{N}\leq m(\Omega^N)\leq \frac{|\Omega|}{N}.
   \end{align*}
   Applying these estimates to bound $\big(1-m(\Omega^N)\big)^{N+1}$ from above and below and taking limits as $N\to\infty$ in both sides, we get
   \[e^{-|\Omega|} \le \liminf_{N\to\infty}  \big(1-m(\Omega^N)\big)^{N+1} \le \limsup_{N\to\infty} 
   \big(1-m(\Omega^N)\big)^{N+1} \le e^{-|\Omega_R|}\]
   Taking limits as $R\to\infty$, the first claim follows.

   For any $z\in\C$ let $B=\mathbb{D}_\rho(z)$ with $\pi \rho^2=|\Omega|$.
   Then $B^N=\mathbb D_{ \sqrt{\pi/N} \rho}(z_N)$, where $z_N=\sqrt{\frac{\pi}{N}}z$.
   In addition, let $\mathcal B_N=\mathbb D_{r_N}(z_N)$  such that $m(\mathcal B_N)=m(\Omega^N)$.
   
   Let
    \begin{align*}
        A_N&=\frac{m(\Omega^N\backslash B^N)+m(B^N\backslash \Omega^N)}{m(\Omega^N)},
    \end{align*}
    which looks like $\mathcal A_m(\Omega^N, B^N)$ except for the uncertainty about $m(\Omega^N)=m(B^N)$.
   Since $\mathcal B_N$ and $B^N$ are concentric discs,  one may observe that
    \begin{align*}
        \left|\mathcal A_m(\Omega^N, \mathcal B_N)-A_N\right|
        \leq\frac{|m(\mathcal B_N)-m(B^N)|}{m(\Omega^N)}=\left|1-\frac{m(B^N)}{m(\Omega^N)}\right|.
    \end{align*}
    Arguing as for the first claim, one gets $\frac{m(B^N)}{m(\Omega^N)}\to\frac{|B|}{|\Omega|}$ as $N\to\infty$ and hence,
   \begin{align*}
        \lim_{N\to\infty} &\left|\mathcal A_m(\Omega^N, \mathcal B_N)-A_N\right|=0,\\
       \lim_{N\to\infty} &A_N=\frac{2|\Omega\backslash B|}{|\Omega|}.
   \end{align*}

   Therefore
   \begin{align*}
        \lim_{N\to\infty} \mathcal A_m(\Omega^N, \mathcal B_N)=\frac{2|\Omega\backslash B|}{|\Omega|}.
   \end{align*}
   Taking the infimum in $z\in\C$ the second claim is shown. 
\end{proof}

In the Fock space, the concentration operator for any measurable set $\Omega$ is given by
\begin{align*}
    C_\Omega(f):=\frac{\int_\Omega |f(z)|^2 e^{-\pi|z|^2}dz}{\|f\|^2_{\mathcal F^2}}, \quad f\in\mathcal F^2.
\end{align*}

We also define the distance of any $f\in\mathcal F^2$ with ${\|f\|^2_{\mathcal F^2}}=1$ to the normalized reproducing kernels given by $\kappa_a(z)=e^{-\pi|a|^2/2} e^{\pi \bar a z}$ as
\begin{align*}
    D(f)&=\min\{\|f-e^{i\theta}f_a\|_{\mathcal F^2}: a\in\C, \theta \in [0,2\pi]\}.
\end{align*}

Now we obtain $C_\Omega(P)$ and $D(P)$ as the limit of $C_{N,\Omega}$ and $D_N$  for suitable rescaled polynomials and domains. 
\begin{lemma}\label{lem:limFock3}
Let $P\in\mathcal P_M$  with $\|P\|_{\mathcal F^2}=1$ and let $\Omega\subset\C$ be measurable
with finite Lebesgue measure.
Then 
\begin{align*}
    \lim_{N\to \infty} C_{N, \Omega^N} (P^N) &= C_\Omega(P),\\
    \lim_{N\to \infty} D_N(\hat P^N) &= D(P),\\
\end{align*}
where $\hat P^N(z)=\frac{P^N(z)}{\|P^N\|_{N}}$.
\end{lemma}

\begin{proof}
    By Lemma~\ref{lem:valuemin} and an analog result in the Fock space (see Lemma 2.5 in \cite{GGRT})
   \begin{align*}
       D_N(\hat P^N)^2&=2\left(1-\frac{1}{\|P^N\|_{N}}{\sup_{z\in\C}\frac{|P^N(z)|}{(1+|z|^2)^{N/2}}}\right)\\
       &=2\left(1-\frac{1}{\|P^N\|_{N}}{\sup_{z\in\C}\frac{|P(z)|}{\left(1+\frac{\pi|z|^2}{N}\right)^{N/2}}}\right),\\
       D(P)^2&=2\left(1-{\sup_{z\in\C}|P(z)|e^{-\pi|z|^2/2}}\right).
   \end{align*}
   
   Then, $\lim_{N\to\infty}D_N(\hat P^N)^2=D(P)^2$ by Lemma~\ref{lem:limFock1} and the monotonicity of the functions involved, which allows to interchange the order of the limit and the supremum.

   The limit for the concentration follows from Lemma~\ref{lem:limFock1} and
   \begin{align*}
      (N+1) \int_{\Omega^N} \frac{|P^N(z)|^2}{\pi(1+|z|^2)^{N+2}}dz=\frac{N+1}{N}\int_\Omega \frac{|P(z)|^2}{\left(1+\frac{\pi|z|^2}{N}\right)^{N+2}}dz
      \xrightarrow[N\to\infty]{} \int_\Omega |P(z)|^2 e^{-\pi|z|^2}dz.
   \end{align*}
\end{proof}

\begin{cor}\label{cor:concentration_Fock}
    There exists a constant $C>0$ such that for any measurable set $\Omega\subset\C$ with positive Lebesgue measure and any  $f\in\mathcal F^2$ with $\|f\|_{\mathcal F^2}=1$, there holds
    \begin{align*}
        C_\Omega(f)\leq \left(1-Ce^{-|\Omega|}D(f)^2\right)C_{\Omega^*}(1),
    \end{align*}
    where $\Omega^*$ is the disc centered at $z=0$ with $|\Omega|=|\Omega^*|$.
    Equivalently,
   \begin{align*}
       D(f)\leq \big(C^{-1} e^{|\Omega|}\delta(f,\Omega)\big)^{1/2},
    \end{align*}
    where
    \begin{align*}
    \delta (f,\Omega) &=1-\frac{C_\Omega(f)}{C_{\Omega^*}(1)}=1-\frac{\int_\Omega |f(z)|^2e^{-\pi|z|^2}dz}{1-e^{-|\Omega|}}.
    \end{align*}
 Moreover, 
    \begin{align*}
        \mathcal A(\Omega)\leq C\frac{e^{3|\Omega|/2}}{|\Omega|}\delta(P,\Omega)^{1/2},
    \end{align*}
\end{cor}

\begin{proof}
    If $f$ is a polynomial, the estimates follow from the combination of Theorem~\ref{thm:concentration} and Proposition~\ref{prop:domains} with the limits in Lemmas~\ref{lem:limFock2} and~\ref{lem:limFock3}. By density, the general results hold. 
\end{proof}

Finally, we deal with the generalized Wehrl entropy in $\mathcal F^2$, given by
\begin{align*}
    S_\Phi(f)=-\int_\C\Phi\big(|f(z)|^2e^{-\pi|z|^2}\big)dz,
\end{align*}
where $\|f\|_{\mathcal F^2}=1$ and $\Phi:[0,1]\to\mathbb R$ is a convex, non-linear, continuous function with $\Phi(0)=0$.

\begin{lemma}\label{lem:limFock4}
Let $\Phi:[0,1]\to\mathbb R$ be a convex, non-linear, continuous function with $\Phi(0)=0$ and let $P\in\mathcal P_M$  with $\|P\|_{\mathcal F^2}=1$.
Then 
\begin{align*}
    \lim_{N\to \infty} S_{N,\Phi}(P^N) &= S_\Phi(P).
\end{align*}
\end{lemma}

\begin{proof}
    After a suitable change of variables, we obtain
    \begin{align*}
        S_{N,\Phi}(P_N)=-\int_\C \Phi\left(u_N(z)\right)\frac{1}{\left(1+\frac{\pi|z|^2}{N}\right)^2}dz,
    \end{align*}
    where 
    \begin{align*}
        u_N(z)=\frac{|P^N(z)|^2}{\|P^N\|_{N}^2\left(1+\frac{\pi|z|^2}{N}\right)^N}.
    \end{align*}
    Notice that $\lim_{N\to\infty} u_N(z)=u(z)=|P(z)|^2 e^{-\pi|z|^2}$. 
    Now we observe
    \begin{align*}
        \left|S_{N,\Phi}(P^N)-S_\Phi(P)\right|\leq \int_\C\frac{|\Phi(u_N)-\Phi(u)|}{\left(1+\frac{\pi|z|^2}{N}\right)^2}dz+\int_\C|\Phi(u)|\left(\frac{1}{\left(1+\frac{\pi|z|^2}{N}\right)^2}-e^{-\pi|z|^2}\right)dz.
    \end{align*}
    Using the continuity of $\Phi$ and the dominated convergence theorem, it follows that the right-hand side goes to $0$ as $N\to\infty$. 
\end{proof}

\begin{cor}\label{cor:entropy_Fock}
    Let $\Phi:[0,1]\to\mathbb R$ be a convex, non-linear, continuous function with $\Phi(0)=0$. Then there exists a constant $C>0$ such that for any $f\in\mathcal F^2$  with $\|f\|_{\mathcal F^2}=1$ it holds
    \begin{align*}
        D(f)^2\leq C\left(S_\Phi(f)-S_\Phi(1)\right).
    \end{align*}
\end{cor}

\begin{proof}
    The result can be proven as Corollary~\ref{cor:concentration_Fock} using also Lemma~\ref{lem:limFock4}.
\end{proof}

\section{Sharpeness}
\label{sec:optimality}

In this section we study the sharpness of   Theorems~\ref{thm:concentration} and~\ref{thm:Wehrl} and Proposition~\ref{prop:domains}  regarding the powers of $\delta_N(P, \Omega)$ and the extra dependence on $m(\Omega)$. We can inherit this from the sharpness of the corresponding inequalities in the Fock space, which were proved in \cite{GGRT} and \cite{FNT23}.

Indeed, the factor $\delta_N(P,\Omega)^{1/2}$ cannot be replaced by 
$\delta_N(P,\Omega)^{\alpha}$ with $\alpha>\frac 12$ independent of $N$ in 
Theorem~\ref{thm:concentration} and Proposition~\ref{prop:domains}. Otherwise, 
arguing as in Corollary~\ref{cor:concentration_Fock}, we would obtain the same dependence 
in the case of Fock, contradicting  Corollary 6.2 in \cite{GGRT}. 
Similarly, in  Theorem~\ref{thm:concentration}, we could not substitute 
$1-m(\Omega)$ by $1-cm(\Omega)$ with $c<1$.

Furthermore, in \cite{FNT23} it is claimed that  in Corollary~\ref{cor:entropy_Fock} we cannot replace $D(f)^2$ by $D(f)^\alpha$ with $\alpha>2$. This can be seen taking $f$ as a small perturbation of $1$.
Since the computations are not present, we include here a direct proof of the optimality of Theorem~\ref{thm:Wehrl}.

\begin{prop}
There exists constants $C, C'>0$ such that for any $N\geq 2$ there  exist $p_\varepsilon\in\mathcal P_N$ with $\|p_\varepsilon\|_{N}=1$ such that for $\varepsilon$ small enough
\begin{align*}
    D_N(p_\varepsilon)\geq \frac{C}{N}\varepsilon^2,\\
    S_N(p_\varepsilon)-S_N(1)\leq \frac{C'}{N^2} \varepsilon^4.
\end{align*}
\end{prop}

\begin{proof}
 Consider the function 
$$p_\varepsilon(z) = \frac{1+\varepsilon z}{\sqrt{1+\varepsilon^2/ N}},$$ which
belongs to the space $\mathcal P_N$ and  has
norm 
$\|p_\varepsilon\|_{N} = 1$. We denote by 
$$u_\varepsilon(z) = \frac{|p_\varepsilon(z)|^2}{(1+|z|^2)^N}=\frac{1}{c_\varepsilon}\frac{|1+\varepsilon z|^2}{(1+|z|^2)^N}.$$

\noindent\emph{Step 1: Computation of $D_N(p_\varepsilon)$.} By Lemma~\ref{lem:valuemin}, $$D_N(p_\varepsilon)^2=2(1-\sqrt{T})=2(1-\sqrt{\sup_{z\in\C} u_\varepsilon(z)}).$$
It can be seen that the  maximum of $u_\varepsilon$ is attained at 
  \begin{align*}
     z_0=\frac{\sqrt{N^2+4\varepsilon^2(N-1)}-N}{2\varepsilon (N-1)}=\frac{\varepsilon}{N}+o(\varepsilon^2).
 \end{align*} 
Hence
 \begin{align*}
     T &= u_\varepsilon(z_0)=\frac{1}{c_\varepsilon}\frac{(1+\varepsilon z_0)^2}{(1+z_0^2)^N} 
     =1-\frac{N-1}{2N^3}\varepsilon^4+O(\varepsilon^6).
 \end{align*} 
 Therefore
 \begin{align}\label{eq:DPep}
     D_N(p_\varepsilon)^2=\frac{N-1}{2N^3}\varepsilon^4+O(\varepsilon^6).
 \end{align}

\noindent\emph{Step 2: Estimation of $S_N(p_\varepsilon)$.} 
We split the entropy in three terms as follows
\begin{align*}
   S_N(p_\varepsilon)&
   =-(N+1)\int_{\mathbb C} u_\varepsilon(z) \log (u_\varepsilon(z))  dm(z)\\
   &=(N+1)\int_{\mathbb C} u_\varepsilon(z)\left(N\log(1+|z|^2)+\log c_\varepsilon-\log |1+\varepsilon z|^2\right)dm(z)
   \\&=(N+1)(A_\varepsilon+B_\varepsilon-C_\varepsilon).
\end{align*}
The first integral is:
\[
\begin{split}
A_\varepsilon = & \frac{N}{c_\varepsilon} \int_{\C} |1+\varepsilon z|^2\log(1+|z|^2)\frac{dm(z)}{(1+|z|^2)^{N}} = 
      \frac{N}{c_\varepsilon} \int_{\C} (1+\varepsilon^2 |z|^2)\log(1+|z|^2)\frac{dm(z)}{(1+|z|^2)^{N}}  \\
=&\frac{N}{c_\varepsilon}\left(\frac{1}{(N+1)^2}+\varepsilon^2\left(\frac{1}{N^2}-\frac{1}{(N+1)^2}\right)\right)
\end{split}
\]
The second term is:
\[
 B_\varepsilon = \log c_\varepsilon\int_{\C} u_\varepsilon(z)dm(z)=\frac{\log(c_\varepsilon)}{N+1}.
\]
The third term is:
\begin{align*}
C_\varepsilon  
&=\int_{\C} u_\varepsilon(z)\log |1+\varepsilon z|^2
dm(z)\\
&=\int_{\varepsilon|z| < 1/2} u_\varepsilon(z)\log |1+\varepsilon z|^2
dm(z)+\int_{\varepsilon|z| > 1/2} u_\varepsilon(z)\log |1+\varepsilon z|^2
dm(z)
=C_{1,\varepsilon}+C_{2,\varepsilon}
\end{align*}
For the first part we have
\begin{align*}
C_{1,\varepsilon}=  &\frac{1}{c_\varepsilon}\int_{\varepsilon|z| < 1/2} (1+\varepsilon^2|z|^{2} + 2\varepsilon
\Re 
z)\log |1+\varepsilon z|^2 \frac{dm(z)}{(1+|z|^2)^{N}}
\\
=  &\frac{4\varepsilon}{c_\varepsilon}\int_{\varepsilon|z| < 1/2} \Re z\log |1+\varepsilon
z| \frac{dm(z)}{(1+|z|^2)^{N}}.
\end{align*}
The last equality holds because $v(z) = \log|1+\varepsilon z|^2$ is a harmonic
function in the disc $\varepsilon |z| < 1/2$, with $v(0) = 0$ and we use the mean
value property in each circle.

Since $\log |1+\varepsilon z| = \sum_{n= 1}^\infty \frac{(-1)^{n+1}\varepsilon^n}{n}\Re
(z^{n})$
uniformly in the disc $|\varepsilon z|< 1/2$, we integrate again in circles and we
get
\[
\begin{split}
 C_{1,\varepsilon} = &  \frac{4\varepsilon^2}{c_\varepsilon} \int_{\varepsilon|z| < 1/2}(\Re z)^2
\frac{dm(z)}{(1+|z|^2)^{N}}  \\
 = & \frac{2\varepsilon^2}{c_\varepsilon} \int_{\C}|z|^2 \frac{dm(z)}{(1+|z|^2)^{N}} -\frac{2\varepsilon^2}{c_\varepsilon} \int_{\varepsilon|z| > 1/2}|z|^2 \frac{dm(z)}{(1+|z|^2)^{N}} =
\frac{2\varepsilon^2}{N(N+1)c_\varepsilon}-C_{3,\varepsilon}
\end{split}
\]
Finally, we can see that both $C_{2,\varepsilon}$ and $C_{3,\varepsilon}$ decrease faster than $\varepsilon^4$ as  $\varepsilon\to 0$ for $N \geq 2$.
This is based on 
\begin{align*}
    \iota_M(\varepsilon)=
    \int_{\varepsilon|z|>1/2}\frac{1}{(1+|z|^2)^M}dz=c_M \varepsilon^{2(M-1)}\big(1+o(1)\big), \quad M\geq 1.
\end{align*}
Indeed, 
\begin{align*}
    C_{2,\varepsilon}&\leq \frac{C}{c_\varepsilon}\int_{\varepsilon|z|>1/2}\frac{|1+\varepsilon z|^4}{(1+|z|^2)^{N+2}}dz
    \leq  \frac{C}{c_\varepsilon}\varepsilon^4  \int_{\varepsilon|z|>1/2}\frac{|z|^4}{(1+|z|^2)^{N+2}}dz \leq  \frac{C}{c_\varepsilon}\varepsilon^4 \iota_{N}(\varepsilon) =o(\varepsilon^5),
    \\
    C_{3,\varepsilon}&\leq \frac{C}{c_\varepsilon} \varepsilon^2\int_{\varepsilon|z|>1/2}\frac{| z|^2}{(1+|z|^2)^{N+2}}dz = \frac{C}{c_\varepsilon} \varepsilon^2\iota_{N+1}(\varepsilon) =o(\varepsilon^5).
\end{align*}

Putting all together 
\begin{align*}
    S_N(p_\varepsilon)&=
 (N+1)(A_\varepsilon + B_\varepsilon - C_\varepsilon)\\
 &=\log(c_\varepsilon)+\frac{1}{c_\varepsilon}\left(\frac{N}{N+1}-\frac{\varepsilon^2}{N(N+1)}\right)+o(\varepsilon^4)\\
 &=\frac{N}{N+1}+\frac{1}{2 N^2}\varepsilon^4+o(\varepsilon^4)
\end{align*}
Therefore,
\begin{align}\label{eq:Spe}
    S_N(p_\varepsilon)-S_N(1)=\frac{1}{2 N^2}\varepsilon^4+o(\varepsilon^4).
\end{align}

The conclusion follows from \eqref{eq:DPep} and \eqref{eq:Spe}.

\end{proof}

\section{Stability of general operators}
\label{sec:generalop}

Theorem~\ref{thm:concentration}, Proposition~\ref{prop:domains} and Theorem~\ref{thm:Wehrl} remain valid when we consider general operators instead of simple projections to the space generated by some $P\in \mathcal P_N$.
Namely, let $\rho: \mathcal P_N\to \mathcal P_N$ be a positive-semidefinite operator such that $\Tr \rho=1$. 
Both the concentration operator in $\Omega$ and the (generalized) Wehrl entropy for $\rho$ can be defined as 
\begin{align*}
    C_{N, \Omega}[\rho]&=\frac{\int_\Omega u(z) dm(z)}{\int_\C u(z) dm(z)},
    \\
    S_{N,\Phi}[\rho]&=-(N+1)\int_\C \Phi\left(u(z)\right) dm(z),
\end{align*}
where 
\begin{align*}
    u(z)=\langle \kappa_{N,z}, \rho (\kappa_{N,z})\rangle_N
\end{align*} 
and $\Phi$ is as in Theorem~\ref{thm:Wehrl}.

Notice that if $\rho$ is the projection on $P\in\mathcal P_N$, with $\|P\|_N=1$, i.e.
$\rho(q)=\langle q,P\rangle_N P$, then both quantities agree with the ones defined in Section~\ref{sec:intro}, since
\begin{align*}
    u(z)=\big\langle \kappa_{N,z}, \langle \kappa_{N,z}, P\rangle_N P \rangle_N=|\langle P,\kappa_{N,z}\big\rangle_N |^2=\left|\frac{P(z)}{(1+|z|^2)^{N/2}}\right|^2
\end{align*}

The concentration $C_{N,\Omega}[\rho]$ achieves its maximum among all sets of measure $\ell$ and all operators as described above when $\Omega$ is a disc with $m(\Omega)=\ell$ and $\rho$ is the projection on the normalized reproducing kernel at the (chordal) center of the disc, i.e.
\[\rho(q)= \Pi_{\kappa_{N,a}}(q)= \langle q,\kappa_{N,a}\rangle_N\ \kappa_{N,a}=\frac{q(a)}{(1+|a|^2)^{N/2}}\kappa_{N,a}, \quad q\in\mathcal P_N.\]
In particular, this happens if $\Omega$ is the disc centered at the origin and $\rho=1$, understood as the identity operator.
In turn, the (generalized) Wehrl entropy attains its minimum  for the same kind of operators.

From Sections~\ref{sec:superlevelsets} and~\ref{sec:proofs}, with mild adaptations, 
we can estimate how close an operator is to the projections on the normalized reproducing kernels in terms of the distance of  its concentration or entropy to their critical values. 
In order to do so, we define the following distance for any positive-semidefinite operator $\rho$ with $\Tr \rho=1$:
\begin{align*}
    D_N[\rho]=\min\{\|\rho-\Pi_{\kappa_{N,a}}\|_1: a\in\C\},
\end{align*}
where $\|\rho\|_1=\Tr |\rho|$.

\begin{thm}\label{thm:concentration_generalop}
    There exists a constant $C>0$ such that for any measurable set $\Omega\subset\C$ with positive  measure and any  positive-semidefinite operator operator $\rho:\mathcal P_N\to \mathcal P_N$  with $\Tr \rho=1$, there holds
   \begin{align*}
       D_N[\rho]\leq \left(C \left(1-m(\Omega)\right)^{N+1}\delta_N[\rho,\Omega]\right)^{1/2},
    \end{align*}
    where
    \begin{align*}
    \delta_N [\rho,\Omega] &=1-\frac{C_\Omega[\rho]}{C_{\Omega^*}[1]}.
    \end{align*}
 Moreover, 
    \begin{align*}
        \mathcal A_m(\Omega)\leq C\frac{\left(1-m(\Omega)\right)^{-3(N+1)/2}}{m(\Omega)}\delta_N[\rho,\Omega]^{1/2},
    \end{align*}
\end{thm}

\begin{thm}\label{thm:entropy_generalop}
    Let $\Phi:[0,1]\to\mathbb R$ be a convex, non-linear, continuous function with $\Phi(0)=0$. Then there exists a constant $C>0$ such that for any  positive-semidefinite operator $\rho:\mathcal P_N\to \mathcal P_N$  with $\Tr \rho=1$ it holds
    \begin{align*}
        D_N[\rho]^2\leq C\left(S_{N,\Phi}[\rho]-S_{N,\Phi}[1]\right).
    \end{align*}
\end{thm}

Once more, these two results give us the analog ones in the Fock space, obtained in \cite{FNT23}, when $N\to\infty$.

There is an interpretation of our results on concentration of 
operators and estimates of the Wehrl entropy in the formalism of quantum mechanics:

Given a collection of $P_j\in \mathcal P_N$ with $\|P_j\|_N=1$ (not 
necessarily pairwise orthogonal) and a sequence of weights $w_j$ with the 
property that $0\le w_j\le 1$ and $\sum_j w_j = 1$, we define the density 
operator $\rho: \mathcal P_N \to \mathcal P_N$ of a mixed state as
\[
 \rho(q) = \sum_j w_j \langle q, P_j\rangle_N P_j.  
\]
This representation is not unique and it defines a positive-semidefinite 
operator with $\Tr \rho = 1$. When the rank of $\rho$ is equal to one, i.e., 
$\rho(q) = \langle q, P\rangle_N P$ with $\|P\|= 1$, then $\rho$ is the density 
operator defining a pure state. If $P$ is a normalized reproducing kernel, we 
have a Bloch coherent state. Thus, our results provide a quantification of the fact 
that coherent states minimize the Wehrl entropy and maximize the concentration 
among all mixed states. 
For further details, the interested reader may consult for instance \cite{Sch22} and references therein.

\begin{proof}[Proofs of Theorem~\ref{thm:concentration_generalop} and~\ref{thm:entropy_generalop}]
    As we have announced, these proofs work as the corresponding ones in the case of polynomials instead of operators. Nevertheless, some adaptations must be done, which we collect below. 

First of all, we notice that since $\rho$ is a positive-semidefinite operator with $\Tr\rho=1$, then there exists an  orthonormal basis $\{P_j\}_{j=0}^N$ of $\mathcal P_N$  and $w_j\geq 0$ with $\sum_{j=0}^N w_j=1$ such that $\rho$ can be written as
\begin{align}\label{eq:rhoq}
    \rho(q)=\sum_{j=0}^N w_{j} \langle q, P_j\rangle_N P_j.
\end{align}
Let $J=\{j\in\{0,\dots, N\}: w_j\neq 0\}$.
Therefore,
\begin{align}\label{ref:u_generalop}
    u(z)&=\sum_{j\in J} w_j\langle P_j,\kappa_{N,z}\rangle_N \langle\kappa_{N,z}, P_j\rangle_N
    = \frac{\sum_{j\in J} w_j|P_j(z)|^2}{\left(1+|z|^2\right)^N}.
\end{align}

We now observe that most of Sections~\ref{sec:superlevelsets} and~\ref{sec:proofs} deal directly with  the function $\mu(t)$ and not with the particular expression for $u(z)$. 
Consequently, we only need to review those parts where the expression for $u(z)$ has a major role and verify whether the conclusions there do not differ from the ones in the simple case where $J$ has one single element.

\noindent\textit{Lemma~\ref{lemma:super-level}.}
The first point  where we have to work directly with $u(z)$ is in Lemma~\ref{lemma:super-level}, and more precisely, in the first two steps. 

\noindent\textit{Step 1:}
Let $P_j(z)=\sum_{n=0}^N p_{j,n} e_n(z)$, with $\sum_{n=0}^N |p_{j,n}|^2=1$. Without loss of generality, we assume $p_{j,0}$ is a non-negative, real number.
Arguing as above, we can estimate 
\begin{align}\label{eq:estimatePi}
    \begin{split}|P_j(z)|^2
    &\leq p_{j,0}^2+\left(\sum_{n=1}^N |p_{j,n}|^2\right)\left(\sum_{n=1}^N |e_n(z)|^2\right)+2 p_{j,0}\Re \left(\sum_{n=1}^N p_{j,n}e_n(z)\right)\\
    &\leq p_{j,0}^2+(1-p_{j,0}^2)\left((1+|z|^2)^N-1\right)+2 p_{j,0}\Re \left(\sum_{n=1}^N p_{j,n}e_n(z)\right).
    \end{split}
\end{align}

If we assume $u(z)$ attains its supremum at $z=0$ as in the simple case, we have
\begin{align*}
    \sum_{j\in J} w_j p_{j,0}^2&=T,\\
    \sum_{j\in J} w_j p_{j,0} p_{j,1}&=0.
\end{align*}
The second claim follows from the fact that $\partial_z u(0)=0$ implies $\sum_{j\in J} P_j(0) P_j'(0)=0$.

Combining the previous identities with \eqref{eq:estimatePi}, one gets
\begin{align*}
    \sum_{j\in J} w_j |P_j(z)|^2 &
    \leq T+ (1-T)\left((1+|z|^2)^N-1\right)
    + 2 \sum_{j\in J} w_j p_{j,0}\Re \left(\sum_{n=2}^N p_{j,n}e_n(z)\right).
\end{align*}

\noindent\textit{Step 2:}
Let $\tilde h(z)=\sum_{j\in J}  w_j p_{j,0}\Re \left(\sum_{n=2}^N p_{j,n}e_n(z)\right)$ and $\tilde Q_j(z)=\sum_{n=2}^N p_{j,n}e_n(z)$. 
By the Cauchy-Schwarz inequality, 
\begin{align*}
    |\tilde h(z)|^2 \leq \left(\sum_{j\in J}  w_j p_{j,0}^2\right)\left(\sum_{j\in J}  w_j \left|\sum_{n=2}^N p_{j,n}e_n(z)\right|^2\right)=T \sum_{j\in J} w_j |Q_j(z)|^2.
\end{align*}
We now note that  $|\tilde Q_j(z)|^2$ and their derivatives satisfy the previous estimates for $Q$ with an extra multiplicative term $1-p_{j,0}^2-|p_{j,1}|^2\leq 1-p_{j,0}^2$. Therefore, 
\begin{align*}
    |\tilde h(z)|^2 &\leq \sum_{j\in J} w_j (1-p_{j,0}^2) \frac{N^2}{2}|z|^4\left(1+|z|^2\right)^N\leq (1-T)\frac{N^2}{2}|z|^4\left(1+|z|^2\right)^N.
\end{align*}
Arguing similarly, we infer 
\begin{align*}
    |\partial_r \tilde h(r e^{i\theta})|&\leq \sqrt{1-T} Nr (1+r^2)^{N/2},\\
    |\partial_{rr} \tilde h(r e^{i\theta})|&\leq \sqrt{1-T} \sqrt{2}N^2r (1+r^2)^{N/2}.
\end{align*}

Defining $\varepsilon=\sqrt{1-T}$ and putting all together, we achieve:
\begin{align*}
    u(r e^{i\theta})\leq \frac{T-\varepsilon^2+2\sqrt{T}\varepsilon h(r e^{i\theta})}{(1+r^2)^N}+\varepsilon^2
\end{align*}
with $h$ a harmonic function satisfying \eqref{eq:estimateh},\eqref{eq:estimatehr} and \eqref{eq:estimatehrr}.
Therefore, from now on,  the proof of Lemma~\ref{lemma:super-level} can be continued as in Section~\ref{sec:superlevelsets}.

\noindent\textit{Lemma~\ref{lem:on_mu}:}
Here, the bound for $\Delta v$, where $v=\frac{1}{2} \log u$, plays a key role. We can see that the same bound is satisfied for $u$ as in \eqref{ref:u_generalop}. 

Notice that 
    $$v = \frac 12 \log u = \frac 12\log\left(\sum_{j\in J} w_j |P_j(z)|^2\right) - \frac N2 \log(1+|z|^2).$$
Then, $\Delta_M\log v \ge -2\pi N$ if and only if
    $\Delta \log\left(\sum_{j\in J} w_j |P_j(z)|^2\right) \ge 0$.
    But 
    $$\log \left(\sum_{j\in J} w_j |P_j(z)|^2\right) = \sup_{\theta_j\in [0, 2\pi] }
    \log  \left|\sum_{j\in J} e^{i\theta_j}w_j P_j(z)^2\right|,$$
    and the supremum of subharmonic functions is subharmonic, provided it is upper semi-continuous as in our situation. Thus, it has positive Laplacian.

\noindent\textit{Lemma~\ref{lem:valuemin}:} 
Slightly different from what the lemma states, now we have  $D_N[\rho]\leq  2\sqrt{1-T}$. Notice that despite the small difference in the right-hand side  and that the equality may not hold, the subsequent combination of lemmas and inequalities leads to the same conclusions.

The estimate can be proved  as \cite{Frank23}*{Lemma 6}.
Let $P\in \mathcal P_N$ with $\|P\|_N=1$ and consider $\rho=\langle \cdot, P\rangle_N P$.
Notice that for any $q\in\mathcal P_N$, $\rho(q)-\Pi_{\kappa_{N,a}}(q)\in \mathrm{span}\{P,\kappa_{N,a}\} $.  Therefore, we can restrict our attention to the subspace spanned by  $P$ and $\kappa_{N,a}$.  Working in the orthonormal basis $\{P, P^\perp\}$, the operator $\rho-\Pi_{\kappa_{N,a}}$ takes the form of the matrix
\begin{align*}
    \begin{pmatrix}
        1-|\langle P, \kappa_{N,a}\rangle_N|^2 
        & -\langle P, \kappa_{N,a}\rangle_N \langle \kappa_{N,a}, P^\perp\rangle_N
        \\ -\langle P^\perp, \kappa_{N,a}\rangle_N \langle \kappa_{N,a}, P^\perp\rangle_N
        & -|\langle P^\perp, \kappa_{N,a}\rangle_N|^2 
    \end{pmatrix}
\end{align*}
 This can be diagonalized so in the suitable basis $\rho-\Pi_{\kappa_{N,a}}$ can be expressed as a diagonal  matrix with diagonal elements $\pm |\langle P^\perp, \kappa_{N,a}\rangle_N|=\pm \sqrt{1-|\langle P, \kappa_{N,a}\rangle_N|^2}$.
Hence, 
\begin{align*}
    \|\rho-\Pi_{\kappa_{N,a}}\|_1=2\sqrt{1-\frac{|P(a)|^2}{\left(1+|a|^2\right)^N}}.
\end{align*}
Now, let us consider a general $\rho$ as in \eqref{eq:rhoq}. Then
\begin{align*}
    \|\rho-\Pi_{\kappa_{N,a}}\|_1
    &\leq \sum_{j\in J} w_j \|\langle \cdot, P_j \rangle_N P_j-\Pi_{\kappa_{N,a}}\|_1
    \leq  2 \sum_{j\in J} w_j \sqrt{1-\frac{|P_j(a)|^2}{\left(1+|a|^2\right)^N}}
    \\&\leq 2\sqrt{1-\frac{\sum_{j\in J}w_j|P_j(a)|^2}{\left(1+|a|^2\right)^N}} =2\sqrt{1-u(a)}.
\end{align*}
Therefore, 
\begin{align*}
    D_N[\rho]\leq \min\left\{2\sqrt{1-u(a)}, \ a\in\C\right\}=2\sqrt{1-T}.
\end{align*}

With the observations above and arguing as in Section~\ref{sec:proofs}, the first part of Theorem~\ref{thm:concentration_generalop} and Theorem~\ref{thm:entropy_generalop} follow. It remains to review the second part of Theorem~\ref{thm:concentration_generalop}, which seeks to reproduce Proposition~\ref{prop:domains}.

\noindent{\textit{Proposition~\ref{prop:domains}:}}
We only need to obtain an estimate like \eqref{eq:u-Tu0} with $\varepsilon=\sqrt{1-T}$.
Let $u$ and $P_j$ be as in the step about Lemma~\ref{lemma:super-level} at the beginning of this proof. Recall in particular that $\sum_{j\in J}w_j p_{j,0}^2=T$. Then
\begin{align*}
    Tu_0(z)-u(z)
    &=\sum_{j\in J}w_j\frac{p_{j,0}^2-|P_j(z)|^2}{\left(1+|z|^2\right)^N}
    \leq \sum_{j\in J}w_j\frac{|p_{j,0}-P_j(z)|}{\left(1+|z|^2\right)^{N/2}}
    \frac{p_{j,0}+|P_j(z)|}{\left(1+|z|^2\right)^{N/2}}
    \\&\leq 2\sqrt{T}\sum_{j\in J}\sqrt{w_j}
    \frac{|p_{j,0}-P_j(z)|}{\left(1+|z|^2\right)^{N/2}}
    \leq 2 \left(\sum_{j\in J}{w_j}
    \frac{|p_{j,0}-P_j(z)|^2}{\left(1+|z|^2\right)^{N}}\right)^{1/2}.
\end{align*}
We note  that
\begin{align*}
    \frac{|p_{j,0}-P_j(z)|^2}{\left(1+|z|^2\right)^{N}}\leq \|p_{j,0}-P_j\|_N^2=\sum_{n=1}^N|p_{j,n}|^2=1-p_{j,0}^2,
\end{align*}
and therefore,
\begin{align*}
    Tu_0(z)-u(z)
    &
    \leq 2 \left(\sum_{j\in J}{w_j}
    (1-p_{j,0}^2)\right)^{1/2}=2\sqrt{1-T}=2\varepsilon.
\end{align*}

With this, all the necessary modifications are done and the results follows.
\end{proof}

\section{The Schatten \texorpdfstring{$p$}{p}-norms of the localization operator} 
\label{sec:Schatten}

Given a set $\Omega\subset \C$ the localization operator 
$L_\Omega: \mathcal P_N \to \mathcal P_N$ is defined as
\[
L_\Omega[p](z) = (N+1)\int_{\Omega} \frac{(1+z\bar w)^N}{(1+|w|^2)^N} p(w)\, dm(w).
\]
Equivalently $L_\Omega[p] = \Pi(\chi_\Omega p)$, where $\Pi$ is the orthogonal projection of $L^2\Bigl(\frac{(N+1)dm(z)}{(1+|z|^2)^N}\Bigr)$ to its subspace $\mathcal P_N$. The integral kernel of the projection is given by $(1+z\bar w)^N$. It is clearly a positive self-adjoint operator with ordered eigenvalues $0 < \lambda_0\le \cdots\le \lambda_N$ and corresponding normalized eigenfunctions $\phi_0, \ldots, \phi_N$. Since the reproducing kernel can be obtained from any orthonormal basis, we have that $\sum_{i= 0}^n \phi_i(z)\overline{\phi_i(w)} = (1+z\bar w)^N$.

It is easily checked that the operator norm $\|L_\Omega\|= \lambda_N$ is given by
\[\|L_\Omega\| = \sup_{p\in \mathcal P_N\setminus\{0\}} C_{N, \Omega}(p).\]
We have already seen that among all sets $\Omega$ with a fixed measure $m(\Omega)$, the disc $\Omega^*$ is the one that gives rise to 
the biggest norm.

One could also consider other norms on the operator. For instance, the Schatten norm $\|L_\Omega\|_p = \left(\lambda_0^p+\cdots + \lambda_N^p\right)^{1/p}$, for $1\le p < \infty$ and $\|L_\Omega\|_\infty = \lambda_N$. When $p=1$, this is known as the trace class norm, for $p=2$ it is the Hilbert-Schmidt norm and for $p=\infty$, it is the operator norm.

We can ask what is the domain that maximizes the Schatten $p$-norm of the concentration operator among all sets with a fixed measure. The case $p=1$, the trace class, is particularly simple because
\begin{align*} 
\sum_{i=0}^N \lambda_i &= \sum_{i=0}^N \langle L_\Omega[\phi_i], \phi_i\rangle_N
\\&=(N+1)^2 \int_{\mathbb C}\sum_{i=0}^N \int_{\Omega} \frac{(1+z\bar w)^N}{(1+|w|^2)^N} \phi_i(w)\, dm(w)\overline{ \phi_i(z)} \frac{dm(z)}{(1+|z|^2)^N}\\
&= (N+1)^2\int_{\Omega} \int_{\mathbb C} \frac{|(1+z\bar w)|^{2N}}{(1+|w|^2)^N(1+|z|^2)^N}\, dm(z) dm(w)
\\&=(N+1) \int_{\Omega} dm(w) = (N+1)m(\Omega).\end{align*}
Thus, $\|L_\Omega\|_1 = (N+1) m(\Omega)$ and hence the norm does not depend on the shape of $\Omega$, only on its mass.

The case $p=2$, the Hilbert-Schmidt norm,  was considered in the context of the Fock space and the Bergman space in \cite{NicRic}. We see now that the same type of result holds in~$\mathcal P_N$. 
The Hilbert-Schmidt norm is given by 
\begin{align*}
\|L_\Omega\|_{HS}^2&=\|L_\Omega\|_2^2= \lambda_0^2+\cdots+\lambda_N^2 = \sum_{i= 0}^N \langle L_\Omega[\phi_i],L_\Omega[\phi_i]\rangle\\
&= (N+1)^3\int_{\mathbb C} \sum_{i=0}^N\int_{\Omega} \frac{(1+z\bar w)^N\phi_i(w)}{(1+|w|^2)^N} \, dm(w) \int_{\Omega} \frac{(1+\bar z \zeta)^N\overline{\phi_i(\zeta)}}{(1+|\zeta|^2)^N} \, dm(\zeta)\, \frac{dm(z)}{(1+|z|^2)^N}\\
&=(N+1)^2 \iint_{\Omega\times \Omega} |(1+w\bar \zeta)|^{2N}\, \frac{dm(\zeta)}{(1+|\zeta|^2)^N}\,\frac{dm(w)}{(1+|w|^2)^N}. 
\end{align*}

Let $d(w,\zeta) = \frac {|w-\zeta|}{\sqrt{\pi(1+|\zeta|^2)(1+|w|^2)}}$ be the chordal distance in the complex plane inherited from the stereographic projection from the North Pole in a sphere of radius $\frac{1}{2\sqrt{\pi}}$.
Then
\[\frac{|(1+\bar\zeta w)|^{2N}}{(1+|\zeta|^2)^N(1+|w|^2)^N} = \left(1-\pi d^2(\zeta,w)\right)^N = \varphi\left(d(\zeta,w)\right),\]
where $\varphi(t) = (1-\pi t^2)^N$ is a decreasing function in $\left[0,\frac{1}{\sqrt{\pi}}\right]$.

We use now the following spherical version of the Riesz rearrangement inequality, a proof of which can be found in \cite{Baernstein}*{Corollary 7.1}:
\begin{thm}\label{Riesz} Let $f$ and $g$ be nonnegative measurable functions on $\mathbb S^n$ and let $\varphi:\mathbb R^+\to \mathbb R^+$ be decreasing. Then
\[
\iint_{\mathbb S^n \times\mathbb S^n } f(x)g(y)\varphi(d(x,y)) d\sigma(x) d\sigma(y) \le \iint_{\mathbb S^n \times\mathbb S^n }  f^{\#}(x)g^{\#}(y)\varphi(d(x,y)) d\sigma(x) d\sigma(y),
\]
where $\sigma$ is the normalized Lebesgue measure in $\mathbb S^n$ and $f^{\#}$ is the symmetric decreasing rearrangement of $f$ in the sphere.
\end{thm}
\begin{rmk} The statement in \cite{Baernstein}*{Corollary 7.1} is written with  the geodesic distance in the sphere $d_{\mathbb S^2}$ instead of the chordal-arc distance $d$, but since $d = \frac{1}{\sqrt\pi}\sin \frac{d_{\mathbb S^2}}{2}$, then $d$ is an increasing function of $d_{\mathbb S^2}$, and both statements are equivalent. 
\end{rmk}

Since $dm$ is the push-forward measure of $d\sigma$ by the stereographic projection, we can write the Hilbert--Schmidt norm as
$$
\|L_\Omega\|_{HS}^2 = K \iint_{\mathbb S^2 \times\mathbb S^2 } \chi_{A}(x) \chi_{A}(y) \varphi(d(x,y))  d\sigma(x) d\sigma(y),
$$
where $A\subset \mathbb S^2$ is the preimage of $\Omega \subset \C$ by the stereographic projection, $\varphi(t) =  (1-t^2)^N$, and $d(x,y)$ is the chordal  distance from $x$ to $y$. 
An immediate consequence of Theorem~\ref{Riesz} is that 
$\|L_\Omega\|_{HS}^2 \le  \|L_{\Omega^*}\|_{HS}^2$, where $\Omega^*$ is a disc such that
$m(\Omega) = m (\Omega^*)$. 
Thus, the behavior of the $\|L_\Omega\|_p$ norm is the same when $p=2$ and $p=\infty$, while for $p=1$ the problem becomes trivial. Other $p$-Schatten norms are more difficult to analyze due to the lack of integral expressions.

It is also possible to study a quantitative version of such inequality, as in \cite{NicRic}, but this may be the content of future work.

\DefineSimpleKey{bib}{archiveprefix}{}

\BibSpec{arXiv}{
  +{}{\PrintAuthors}{author}
  +{,}{ \textit}{title}
  +{}{ \parenthesize}{date}
  +{,}{ arXiv }{eprint}
}

\end{document}